\documentclass[11pt,hidelinks,a4paper]{article}

\usepackage{algpseudocode}
\usepackage{algorithm}
\usepackage{amsfonts}
\usepackage{amsmath}
\usepackage{amssymb}
\usepackage{amstext}
\usepackage{amsthm}
\usepackage[toc,page]{appendix}
\usepackage{authblk}
\usepackage{booktabs}
\usepackage{cite}
\usepackage{datetime}
\usepackage{enumitem}
\usepackage{framed}
\usepackage{fullpage}
\usepackage{graphics}
\usepackage{graphicx}
\usepackage[latin1]{inputenc}
\usepackage{mathrsfs}
\usepackage{multirow}
\usepackage{pdflscape}
\usepackage{pgfplots}
\usepackage{pifont}
\usepackage{rotating}
\usepackage{setspace}
\usepackage{subcaption}
\usepackage{tikz}
\usepackage{times}
\usepackage{units}
\usepackage{url}
\usepackage{xcolor}
\usepackage{hyperref}
\usepackage[font=small]{caption}

\usepackage[textsize=small,color=blue!20]{todonotes}

\usetikzlibrary{shapes,arrows,plotmarks,matrix,positioning,fit,calc,3d}
\newcommand{\iid}{\stackrel{\mathrm{iid}}{\sim}}

\newtheorem{theorem}{Theorem}[section]
\newtheorem{lemma}[theorem]{Lemma}

\newtheorem{corollary}[theorem]{Corollary}
\newtheorem{remark}[theorem]{Remark}
\newtheorem{mydef}[theorem]{Definition}
\newtheorem*{assumptions}{Assumptions}

\DeclareMathOperator*{\argmin}{argmin}

\newcounter{num}
\renewcommand{\thenum}{\arabic{num}}

\newlength\figureheight
\newlength\figurewidth

\begin{document}

\title{Convergence and Rates for Fixed-Interval Multiple-Track\\ Smoothing Using
  $k$-Means Type Optimization}
%On the Asymptotic Behavior of a Variational Approach to Post-Hoc Track Estimation with Data Association %On the Asymptotic Behavior of a Variational Approach to Post-Hoc Estimation of Multiple Tracks
\author{Matthew Thorpe}
\affil{Department of Mathematical Sciences, Carnegie Mellon University,\\ Pittsburgh, PA 15213, United States}
\author{Adam M. Johansen}
\affil{Department of Statistics, University of Warwick, Coventry, CV4 7AL, United Kingdom}
%\date{\today \; at \currenttime}
\date{}
\maketitle

\begin{abstract}
We address the task of estimating multiple trajectories from unlabeled
data. This problem arises in many settings, one could think of the
construction of maps of transport networks from passive observation of
travellers, or the reconstruction of the behaviour of uncooperative vehicles
from external observations, for example.
There are two coupled problems.
The first is a data association problem: how to map data points onto individual trajectories.
The second is, given a solution to the data association problem, to estimate those trajectories.
We construct estimators as a solution to a regularized variational problem (to
which approximate solutions can be obtained via the simple, efficient and
widespread $k$-means method) and show that, as the number of data points, $n$, increases, these estimators exhibit stable behaviour.
More precisely, we show that they converge in an appropriate Sobolev space in probability and with rate $n^{-1/2}$.
\end{abstract}

\section{Introduction \label{sec:Intro}}

Given observations from multiple moving targets we face two (coupled) problems.
The first is associating observations to targets: the data association problem.
The second is estimating the trajectory of each target given the appropriate set of observations.
When there is exactly one target the data association problem is trivial. 
However, when the number of targets is greater than one (even when the number of targets is known) the set of data association hypotheses grows combinatorially with the number of data points.
Very quickly it becomes infeasible to check every possibility.
Hence fast approximate solutions are needed in practice.

In this paper we interpret the problem of estimating multiple trajectories with unknown data association (see Figure~\ref{fig:Intro:SDAFig}) in such a way that the $k$-means method~\cite{lloyd82} may be applied to find a solution.
As in~\cite{thorpe15}, this is a non-standard application of the $k$-means
method in which we generalize the notion of a `cluster center' to partition finite dimensional data using infinite dimensional cluster centers. 
In this paper the cluster centers are trajectories in some function space and the data are space-time observations.

\begin{figure}
\centering
\setlength\figureheight{4cm}
\setlength\figurewidth{8cm}
\input{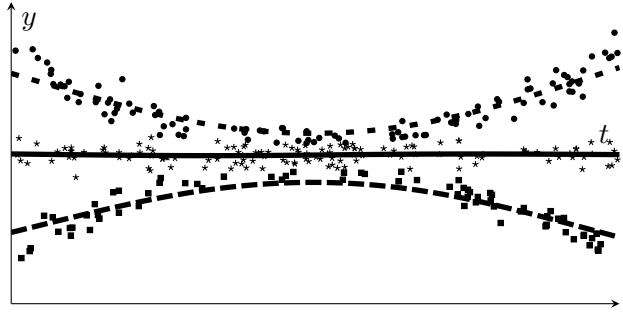}
\caption{
Unlabeled data is generated from three targets and using minimizers of~\eqref{eq:Intro:fn} we can find a partitioning of the data set and nonparametrically estimate each trajectory using the $k$-means algorithm.
}
\label{fig:Intro:SDAFig}
\end{figure}

Let $\Theta\subset (H^s)^k$ where $H^s$ is the Sobolev space of degree $s$ (where we
consider the case $s\geq 1$, see Section \ref{subsec:Prelim:Not} for a precise definition).
We have a data set $\{(t_i,y_i)\}_{i=1}^n\subset [0,1]\times \mathbb{R}^d$ and a model for the observation process
\begin{equation} \label{eq:Intro:model}
y_i = \mu_{\varphi(i)}^\dagger(t_i) + \epsilon_i
\end{equation}
where $\mu^\dagger = (\mu^\dagger_1,\dots,\mu^\dagger_k)$ is some unknown function, $\epsilon_i\iid\phi_0$ and $t_i\iid\phi_T$ for densities $\phi_0$ and $\phi_T$ on $[0,1]$ and $\mathbb{R}^d$ respectively.
We assume that the index of the cluster responsible for any given observation is an independent random variable with a categorical distribution of parameter vector $p = (p_1,\ldots,p_k)$, writing $\varphi(i) \sim\textrm{Cat}(p)$ to mean $\mathbb{P}(\varphi(i) = j)=p_j$.
This assumptions allow us to write the density of $y$ given $t$ (and, implicitly, the cluster centres), which we denote by $\phi_Y(y|t)$, as
\[ \phi_Y(y|t) = \sum_{j=1}^k p_j \phi_0(y-\mu^\dagger_j(t)). \]
We can summarize the stylized data generating process as follows.
A cluster is selected at random: $\mathbb{P}(\varphi = j)=p_j$, the time and observation error are drawn independently from their respective distributions, $t \sim \phi_T$, and $\epsilon \sim \phi_0$; and we observe $(t,y=\mu^\dagger_\varphi(t)+\epsilon)$.

The aim is to estimate $\mu^\dagger=(\mu^\dagger_1,\dots,\mu^\dagger_k)\in \Theta$ from observed data $\{(t_i,y_i)\}_{i=1}^n$.
In particular the data association
\[ \varphi:\{1,2,\dots,n\}\to\{1,2,\dots,k\} \]
is unknown.
With a single trajectory ($k=1$) the problem is precisely the spline smoothing problem, see for example \cite{wahba90}.
For $k>1$ trajectories there is an additional data association problem coupled to the spline smoothing problem.
We call this the smoothing-data association (SDA) problem.
Although the estimator $\mu^n$ we propose is not necessarily a consistent estimator for $\mu^\dagger$ (we do not show $\mu^n\to \mu^\dagger$) we do consider our estimator a natural choice.
We believe it is possible to bound the asymptotic error $\lim_{n\to \infty} \|\mu^n - \mu^\dagger \|_{(L^2)^k} \leq C$ where $C$ depends on the distribution of the data points, however it is beyond the scope of this work to show such a bound.
We refer to~\cite[Section 4.5]{levrard15} for a bound of the type $\|\mu^\infty-\mu^\dagger\|\leq C$, where $\mu^\infty = \lim_{n\to \infty} \mu^n$, for $k$-means in Hilbert spaces.

We assume $k$ is fixed and known.
The aim of this paper is to construct a sequence of estimators $\mu^n$ of $\mu^\dagger$ based upon increasing sets of observations $\{(t_i,y_i)\}_{i=1}^n$ and to study their asymptotic behavior as $n\to \infty$.
For each $n$ our estimate is given as a minimizer of $f_n:\Theta\to\mathbb{R}$ defined by
\begin{equation} \label{eq:Intro:fn}
f_n(\mu) = \frac{1}{n} \sum_{i=1}^n \bigwedge_{j=1}^k |y_i - \mu_j(t_i)|^2 + \lambda \sum_{j=1}^k \| \nabla^s \mu_j \|^2_{L^2}
\end{equation}
where $|\cdot |$ is the Euclidean norm on $\mathbb{R}^d$, $\bigwedge_{j=1}^k z_j = \min\{z_1,\dots, z_k\}$ and $\lambda$ is a positive constant.
Penalizing the $s^{\text{th}}$ derivative ensures that the problem is well posed.
Optimizing this function can be interpreted as seeking a hard data association: given $\mu\in\Theta$ each observation $(t_i,y_i)$ is associated with the trajectory closest to it so the corresponding data association solution is given by
\[ \varphi^\mu(i) = \argmin_{j=1,2,\dots,k} |\mu_j(t_i)-y_i|. \]
As with many ill-posed inverse problems with a data association component recovering the `true' values of the (infinite-dimensional) parameters is in general infeasible.
Two approaches are possible: to impose strong parametric assumptions, reducing the problem to that of inferring a (finite-dimensional) collection of parameters (which will perform poorly when those assumptions are inappropriate) or to proceed nonparametrically, optimising a cost function which balances the trade-off between a good fit to the data and regularity of the solution (which requires the precise specification of the notion of regularity).
In this paper we pursue the second route, showing that in the large data limit the proposed estimators behave well.
The main contribution of this paper is to establish the stability of $k$-means like estimators to the SDA problem.   

Although exact solution of the underlying optimization problem is NP-complete even in benign Euclidean settings~\cite{dasgupta09}, the computational cost of iterative numerical approximation has been shown to have a polynomial (smoothed) cost in certain Euclidean settings, e.g. \cite{arthur09}, and in practice the performance is often much better than these bounds would suggest: it is accepted to be a numerically efficient method for obtaining approximate solutions (i.e. local minimizers).
Our empirical experience is that this property holds also within the context considered by this paper.
Our focus is upon the asymptotic properties of the ideal estimator and it is beyond the scope of this paper to  upper bound the computational complexity of the numerical iteration scheme. 
We do however point out that a key advantage of the $k$-means method is that it reduces the problem of solving the multiple target problem ($k>1$) to the problem of repeatedly solving the single target problem ($k=1$) which can be done efficiently with, for example, splines.

There are of course several variations of the $k$-means method, e.g. fuzzy $C$-means clustering~\cite{bezdek13} (a soft version of $k$-means closely-related to the EM algorithm \cite{dempster77}), $k$-medians clustering~\cite{bradley97} (an $L^1$ version of $k$-means), Minkowski metric weighted $k$-means~\cite{deamorim12} for which the analysis, particularly the convergence result in Theorem~\ref{thm:Conv:Conv}, could be easily adapted.
Indeed, for bounded noise, the weak convergence $k$-medians clustering is a special case of~\cite{thorpe15} and to extend the result to unbounded noise one can follow the strategy given in the proof of Theorem~\ref{thm:Conv:Bound}.
The strong convergence and rate of convergence will require a different approach as one loses differentiability when going from $L^2$ to $L^1$.

The choice of regularization scheme and, in particular, of $\lambda$ is not straightforward.
For $k=1$ there are many results in the spline literature on the selection of $\lambda=\lambda_n$ and the resulting asymptotic behavior as $n\to\infty$, see for example~\cite{aerts02,cox83,cox88,craven79,li87,nychka89,ragozin83,speckman85,speckman01,stone82,utreras81,utreras83,utreras85,wand99}.
In this case one has $\lambda_n\to 0$ and can expect $\mu^n$ to converge to $\mu^\dagger$.
Convergence is either with respect to a Hilbert scale, e.g. $L^2$, or in the dual space, i.e. weak convergence.
Using a Hilbert scale in effect measures the convergence in a norm weaker than $H^s$. 
We remark that when $k>1$ and $\lambda_n\to 0$ we would expect that minimizers $\mu^n$ converge to a minimizer $\mu^*$ of
\[ \int_0^1 \int_{\mathbb{R}^d} \bigwedge_{j=1}^k |y-\mu_j(t)|^2 \phi_Y(y|t) \phi_T(t) \, \mathrm{d} y \, \mathrm{d} t. \]
In particular we do not expect that $\mu^*=\mu^\dagger$, indeed even the $k$-means in Euclidean spaces is known to be asymptotically biased.
In this paper we do not take $\lambda_n\to 0$ which adds a further bias.

The approach we take, as is common in settings in which smooth solutions are expected, is to penalize the $s^{\text{th}}$ derivative.
By Taylor's Theorem we can write $H^s = \mathcal{H}_0 \oplus \mathcal{H}_1$ where
\begin{align*}
\mathcal{H}_0 & = \text{span}\left\{ \zeta_i(t)=\frac{t^{i}}{i!} : i=0,1,\dots,s-1 \right\}, \\
\mathcal{H}_1 & = \left\{ g\in H^s : \nabla^i g(0)=0 \text{ for all } i=0,1,\dots,s-1 \right\}.
\end{align*}
We use $\|\cdot\|_1=\|\nabla^s \cdot\|_{L^2}$ as the norm on $\mathcal{H}_1$ and denote the $\mathcal{H}_0$ norm by $\|\cdot\|_0$, and therefore we use the norm $\|\cdot\|_{H^s} = \|\cdot\|_0+\|\cdot\|_1$ on $H^s$ (which is equivalent to the usual Sobolev norm).
Since $\mathcal{H}_0$ is finite dimensional we are free to use any norm we choose without changing the topology.
We can view $H^s=\mathcal{H}_0\oplus\mathcal{H}_1$ as a multiscale decomposition of $H^s$.
The polynomial component represents a coarse approximation.
The regularization penalizes oscillations on the fine scale, i.e. in $\mathcal{H}_1$.

In the case $k=1$, $f_n$ is quadratic and one can find an explicit representation of $\mu^n$, i.e. there exists a random function $G_{n,\lambda}$ such that with probability one $\mu^n=G_{n,\lambda}\nu^n$ for some function $\nu^n$ which depends on the data.
When $k>1$ the problem is no longer convex and the methodology used in the $k=1$ case fails.

The first result of this paper (Theorem~\ref{thm:Conv:Conv}) is a weak
convergence result, we show that there exists $\mu^\infty\in\Theta$ such that (up to subsequences) $\mu^n\rightharpoonup\mu^\infty$ a.s. in $H^s$ and $\mu^\infty$ is a minimizer of $f_\infty$ defined by
\begin{equation} \label{eq:Intro:finfty}
f_\infty(\mu) = \int_0^1 \int_{\mathbb{R}^d} \bigwedge_{j=1}^k |y- \mu_j(t)|^2 \; \phi_Y(y|t) \phi_T(t) \, \mathrm{d} y \, \mathrm{d} t + \lambda \sum_{j=1}^k \| \nabla^s \mu_j\|^2_{L^2}.
\end{equation}
One should note that if $\mu^\infty = (\mu^\infty_1,\dots,\mu^\infty_k)$ is a minimizer of $f_\infty$ then so
is $\tilde{\mu}^\infty = (\mu^\infty_{\rho(1)},\dots,\mu^\infty_{\rho(k)})$ for any permutation $\rho:\{1,\dots,k\} \to
\{1,\dots,k\}$ and therefore we do not expect uniqueness of the minimizer.
Considering the law of large numbers the limit $f_\infty$ is natural.
The functional $f_\infty$ can be seen as a limit of $f_n$, the nature of which will be made rigorous in Section~\ref{sec:Conv}.
The second result is to go from almost sure weak convergence to strong convergence in probability.
In other words, we obtain convergence of the minimizing sequence in a stronger
topology at the expense of considering a weaker mode of stochastic convergence.

We recall that one motivation for considering the minimization problem~\eqref{eq:Intro:fn} is to embed the problem into a framework that allows the application of the $k$-means method.
Large data limits for the $k$-means have been studied extensively in finite dimensions, see for example~\cite{antos05,bartlett98,chou94,hartigan78,linder94,pollard81,pollard82,pollard82a}.
There are fewer results for the infinite dimensional case, with~\cite{auder12,cuesta88,cuesta07,fischer10,laloe10,levrard15,lember03,linder02,tarpey03,thorpe15,biau08} the only results known to the authors.
Of these, only~\cite{thorpe15} can be applied to finite dimensional data and infinite dimensional cluster centers but required bounded noise and furthermore the conclusion were limited to weak convergence.
The first contribution of this paper is to extend this convergence result to unbounded noise for the SDA problem (Section~\ref{sec:Conv}).
We point out that \cite{auder12,biau08,laloe10,levrard15} give results for the convergence and rates of convergence of the minimum $\min f_n$ (in infinite dimensional settings) and~\cite{lember03} gives results for the convergence of the minimizers.

The result of Theorem~\ref{thm:Rate:StrongConv} is that, upto subsequences, the convergence is strong in $H^s$.
The final result is to show that the rate of convergence is of order $\frac{1}{\sqrt{n}}$ in probability.
I.e.
\[ \|\mu^n-\mu^\infty\|_{(H^s)^k} = O_p\left(\frac{1}{\sqrt{n}}\right). \]
This is closely related to the central limit theorem first proved for the $k$-means method by Pollard~\cite{pollard82} for Euclidean data.
We extend his methodology to cluster centers in $H^s$ to prove our rate of convergence result and in doing so provide a theoretical justification for using this method in the more complex scenario which we consider and, in particular, for using such approaches to address post hoc tracking of multiple targets using $k$-means type algorithms.
As with Pollard's finite dimensional result we require an assumption on the positive definiteness of the second derivative of the limiting function $f_\infty$.

In the next section we remind the reader of some preliminary material which underpins our main results.
Section~\ref{sec:Conv} contains the weak convergence result.
In Section~\ref{sec:Rate} we go from weak convergence to strong convergence with rates.

\section{Preliminaries \label{sec:Prelim}}

\subsection{Notation \label{subsec:Prelim:Not}}

The Borel $\sigma$-algebra on $[0,1]\times \mathbb{R}^d$ is denoted $\mathcal{B}([0,1]\times\mathbb{R}^d)$ and the set of probability measures on $([0,1]\times\mathbb{R}^d,\mathcal{B}([0,1]\times \mathbb{R}^d))$ by $\mathcal{P}([0,1]\times\mathbb{R}^d)$. 
Our main results concern sequences of data $\{(t_i,y_i)\}_{i=1}^\infty$ sampled independently with common law $P\in\mathcal{P}([0,1]\times \mathbb{R}^d)$ which is assumed to have a Lebesgue density, $\phi((t,y)) = \phi_Y(y|t) \phi_T(t)$.
We work throughout on a probability space $(\Omega,\mathcal{F},\mathbb{P})$ rich enough to support a countably infinite sequence of such observations, $(t_i,y_i):\Omega\to [0,1]\times \mathbb{R}^d$.
All random elements are defined upon this common probability space and all stochastic quantifiers are to be understood as acting with respect to $\mathbb{P}$ unless otherwise stated.
With a small abuse of notation we say $(t_i,y_i)\in [0,1]\times \mathbb{R}^d$.

We will define the space $\Theta\subset (H^s)^k$ in Section~\ref{sec:Conv}.
The Sobolev space $H^s$ is given by
\[ H^s := \left\{ \mu:[0,1]\to \mathbb{R}^d \text{ s.t. } \nabla^i\mu \text{ is absolutely continuous for } i=0,1,\dots,s-1 \text{ and } \nabla^s\mu \in L^2 \right\}. \]
Note that data is of the form $\{(t_i,y_i)\}_{i=1}^n\subset [0,1]\times \mathbb{R}^d$.

We denote weak convergence by $\rightharpoonup$: if $\nu^n,\nu\in H^s$ satisfies $F(\nu^n)\to F(\nu)$ for all $F\in (H^s)^*$ then $\nu^n\rightharpoonup \nu$.
A sequence of probability measures $P_n$ is said to weakly converge to $P$ if for all bounded and continuous functions $h$ we have
\[ P_nh \to P h. \]
Where we write $Ph = \int h(x) \; P(\text{d} x)$.
If $P_n$ weakly converges to $P$ then we write $P_n\Rightarrow P$. 

We use the following standard definitions for rates of convergence.
\begin{mydef}
\label{def:Prelim:Not:Rate}
We define the following.
\begin{itemize}
\item[(i)] For deterministic sequences $a_n$ and $r_n$, where $r_n$ are positive and real valued, we write $a_n=O(r_n)$ if $\frac{a_n}{r_n}$ is bounded.
If $\frac{a_n}{r_n}\to 0$ as $n\to \infty$ we write $a_n=o(r_n)$.
\item[(ii)] For random sequences $a_n$ and $r_n$, where $r_n$ are positive and real valued, we write $a_n=O_p(r_n)$ if $\frac{a_n}{r_n}$ is bounded in probability: for all $\epsilon>0$ there exist deterministic constants $M_\epsilon,N_\epsilon$ such that
\[ \mathbb{P}\left(\frac{\left\vert a_n \right\vert}{r_n}\geq M_\epsilon\right) \leq \epsilon \quad \quad \quad \forall n\geq N_\epsilon. \]
If $\frac{a_n}{r_n}\to 0$ in probability: for all $\epsilon>0$
\[ \mathbb{P}\left(\frac{\left\vert a_n \right\vert}{r_n}\geq \epsilon \right) \to 0 \quad \quad \quad \text{as } n\to \infty \]
we write $a_n=o_p(r_n)$.
\end{itemize}
\end{mydef}

When $a=a(r)$ can be written as a function of $r$ we will often write $a=O(r)$ or $a=o(r)$ to mean for any sequence $r_n\to 0$ that $a_n:=a(r_n)$ satisfies $a_n=O(r_n)$ or $a_n=o(r_n)$ respectively.

\subsection{\texorpdfstring{$\Gamma$}{Gamma}-Convergence \label{subsec:Prelim:Gam}}

Our proof of convergence will use a variational approach.
In particular the natural convergence for a sequence of minimization problems is $\Gamma$-convergence.
The $\Gamma$-limit can be understood as the `limiting lower semi-continuous envelope'.
It is particular useful when studying highly oscillatory functionals when there will often be no strong limit and the weak limit (if it exists) will average out oscillations and therefore change the behavior of the minimum and minimizers.
See~\cite{braides02,dalmaso93} for an introduction to $\Gamma$-convergence and~\cite{trillos15,trillos15b,thorpe15} for applications of $\Gamma$-convergence to problems in statistical inference.
We will apply the following definition and theorem to $\mathcal{H}=\Theta\subset(H^s)^k$. 

\begin{mydef}[$\Gamma$-convergence {\cite[Definition 1.5]{braides02}}]
\label{def:Prelim:Gamma:Gamcon}
Let $\mathcal{H}$ be a Banach space and $\Theta\subset \mathcal{H}$ be a weakly closed set.
A sequence $f_n :\Theta\to \mathbb{R}\cup \{\pm\infty\}$ is said to \textit{$\Gamma$-converge} on $\Theta$ to $f_\infty :\Theta\to \mathbb{R}\cup \{\pm\infty\}$ with respect to weak convergence on $\mathcal{H}$, and we write $f_\infty = \Gamma\text{-}\lim_n f_n$, if for all $\nu\in \Theta$ we have
\begin{itemize}
\item[(i)] (lim inf inequality) for every sequence $(\nu^n)\subset \Theta$ weakly converging to $\nu$
\[ f_\infty(\nu) \leq \liminf_n f_n(\nu^n); \]
\item[(ii)] (recovery sequence) there exists a sequence $(\nu^n)$ weakly converging to $\nu$ such that
\[ f_\infty(\nu) \geq \limsup_n f_n(\nu^n). \]
\end{itemize}
\end{mydef}

When it exists the $\Gamma$-limit is always weakly lower semi-continuous \cite[Proposition 1.31]{braides02} and therefore achieves its minimum on any weakly compact set.
An important property of $\Gamma$-convergence is that it implies the convergence of minimizers.
In particular, we will make use of the following result which can be found in \cite[Theorem 1.21]{braides02}.

\begin{theorem}[Convergence of Minimizers]
\label{thm:Prelim:Gamma:Conmin}
Let $\mathcal{H}$ be a Banach space, $\Theta\subset \mathcal{H}$ be a weakly closed set and $f_n: \Theta\to \mathbb{R}\cup\{\pm\infty\}$ be a sequence of functionals.
Assume there exists a weakly compact subset $K\subset \Theta$ with 
\[ \inf_{\Theta} f_n = \inf_K f_n \quad \forall n\in\mathbb{N}. \]
If $f_\infty = \Gamma\text{-}\lim_n f_n$ and $f_\infty$ is not identically $\pm\infty$ then
\[ \min_{\Theta} f_\infty = \lim_n \inf_{\Theta} f_n. \]
Furthermore if $\mu^n\in K$ minimizes $f_n$ then any weak limit point is a minimizer of $f_\infty$.
\end{theorem}

\subsection{The G\^ateaux Derivative \label{subsec:Prelim:Gat}}

As in Section~\ref{subsec:Prelim:Gam} we will apply the following to $\mathcal{H}=\Theta\subset (H^s)^k$.

\begin{mydef} \label{def:Prelim:GatDer:GatDer}
We say that $f:\mathcal{H}\to \mathbb{R}$ is G\^ateaux differentiable at $\mu\in \mathcal{H}$ in direction $\nu\in \mathcal{H}$ if the limit
\[ \partial f(\mu;\nu) = \lim_{r\to 0^+} \frac{f(\mu+r\nu) - f(\mu)}{r} \]
exists.
We may define second order derivatives by
\[ \partial^2 f(\mu;\nu,\omega) = \lim_{r\to 0^+} \frac{\partial f(\mu+r\omega;\nu) - \partial f(\mu;\nu)}{r} \]
for $\mu,\nu,\omega\in \mathcal{H}$.
In cases where the second derivative does not necessarily exist we will define $\partial_-^2 f$ by
\[ \partial^2_- f(\mu;\nu,\omega) = \liminf_{r\to 0^+} \frac{\partial f(\mu+r\omega;\nu) - \partial f(\mu;\nu)}{r}. \]
To simplify notation, we write:
\[ \partial^2_- f(\mu;\nu):= \partial^2_- f(\mu;\nu,\nu). \]
\end{mydef}

\begin{theorem}
\label{thm:Prelim:GatDer:Taylor}
Let $\mu,\nu\in\mathcal{H}$.
If $f:\mathcal{H}\to \mathbb{R}$ is continuously G\^ateaux differentiable on the set $\left\{ t\mu + (1-t)\nu \, : \, t\in [0,1]\right\}$ then
\[ f(\nu) \geq f(\mu) + \partial f(\mu;\nu-\mu) + \frac12 \partial^2_- f((1-t^*)\mu + t^*\nu;\nu-\mu) \]
for some $t^*\in [0,1]$.
\end{theorem}

\begin{proof}
The theorem is only a slight generalisation of Taylor's theorem.
Indeed, if there exists $t\in [0,1]$ such that $\partial^2_-f((1-t)\mu+t\nu;\nu-\mu) = - \infty$ then we have nothing to prove.
So we assume $\partial^2_-f((1-t)\mu+t\nu;\nu-\mu) > - \infty$ for all $t \in [0,1]$, define $g(t) = f((1-t)\mu+t\nu)$ then we can show that $g(1)=f(\nu)$, $g(0) = f(\mu)$, $g'(0)=\partial f(\mu;\nu-\mu)$ and $g''_-(t) = \partial_-^2 f((1-t)\mu + t \nu;\nu-\mu)$ where we define
\begin{equation} \label{eq:Prelim:GatDer:2ndLD}
g''_-(t) = \liminf_{r\to 0^+} \frac{g'(t+r)-g'(t)}{r}.
\end{equation}
Hence we can equivalently show that $g(1) \geq g(0) + g'(0) +
\frac{1}{2}g''_-(t^*)$ for some $t^*\in[0,1]$.
Define $J = 2(g(1)-g(0)-g'(0))$ and we are left to show $J\geq g_-''(t^*)$.

Let
\[ F(t) = g(t) + g'(t)(1-t) + \frac{(t-1)^2}{2} J - g(1) \]
and note that, by definition of $J$, $F(0) = F(1) = 0$.
Since $F_-'(t) = (1-t)(g_-''(t) - J)$ (where $F_-'$ is defined analogously to~\eqref{eq:Prelim:GatDer:2ndLD}), then if we can show there exists
$t^*\in (0,1)$ such that $F_-'(t^*)\leq 0$ we are done.
One can easily show that if $F_-'(t) > 0$ for all $t$ then $F$ is strictly
increasing, which contradicts $F(1)=F(0)$, and so there must exist such a $t^*$.
\end{proof}

\section{Weak Convergence \label{sec:Conv}}

To show weak convergence we apply Theorem~\ref{thm:Prelim:Gamma:Conmin}. 
The following two subsections prove that the conditions required to apply this theorem, i.e. that $f_\infty$ is the $\Gamma$-limit of $f_n$ and that the minimizers $\mu^n$ are uniformly bounded, hold with probability one.

For a fixed $\delta>0$ we define the set $\Theta$ to be the set of functions in $(H^s)^k$ which have minimum separation distance of $\delta$:
\begin{equation} \label{eq:Conv:Theta}
\Theta = \left\{ \mu\in (H^s)^k : | \mu_j(t) - \mu_l(t)| \geq \delta \; \forall t\in [0,1] \text{ and } j\neq l \right\}.
\end{equation}
For $d=1$ this is a strong assumption as we restrict ourselves to trajectories that do not intersect.
When considering the tracking of real objects in 2 or more dimensions, the assumption is typically physically reasonable.
For example if $\mu_j$ are to represent trajectories of extended objects by
modelling the location of the centroid, it is natural to require a minimum
separation of those centroids on a scale corresponding to the extent of the
objects in question. 

In practical implementations the constraint could be difficult to implement,
but it is straightforward to check whether it is satisfied post hoc.
For  a wide range of distributions on the data it is reasonable to expect that
any two cluster centers obtained by numerical procedures will not intersect and therefore have a minimum separation distance.
Of course, this separation distance is only known with posterior knowledge and not prior knowledge as we assume here.
We expect that one could improve this reasoning to state explicitly that with high probability any two cluster centers are at least $\delta^*$ apart for some $\delta^*$ that depends upon the distribution of the data.
We do not attempt to prove any such statement here.
Such a statement would imply that one could carry out the classification using a $k$-means method without directly imposing the constraint.

We use the assumption in order to infer that the spatial partitioning induced
by any set of cluster centers $\mu\in \Theta$ is such that every element of the partition is non-empty, at every time $t$, i.e. the sets
\[ X_j(t) = \left\{ x\in\mathbb{R}^d \, : \, |x-\mu_j(t)| < |x-\mu_i(t)| \, \text{for } i \neq j \right\} \]
for $j=1,\ldots,k$ are all non-empty.

First let us show that $\Theta$ is weakly closed in $(H^s)^k$.
Take any sequence $\mu^n\in\Theta$ such that $\mu^n\rightharpoonup \mu\in (H^s)^k$.
We have to show $\mu\in\Theta$.
Pick $t\in [0,1]$, $j\neq l$ and define $F:\Theta\rightarrow \mathbb{R}^d$ by $F:\nu \rightarrow \nu_j(t) - \nu_l(t)$, note that $F$ is in the dual space of $(H^s)^k$ (since $s\geq 1$).
Hence
\[ \delta \leq |\mu^n_j(t) - \mu^n_l(t)| = |F(\mu^n)| \to |F(\mu)| = |\mu_j(t) - \mu_l(t)|. \]
Therefore $\mu\in\Theta$.
Furthermore we can show that $f_n,f_\infty$ are weakly lower semi-continuous~\cite[Propositions 4.8 and 4.9]{thorpe15} hence they obtain their minimizers over weakly compact subsets of $\Theta$.
We will show that minimizers are contained in a bounded, and hence weakly compact set, and therefore there exists minimizers of $f_n$ and $f_\infty$ on $\Theta$.

We now state our assumptions.

\begin{assumptions} 
\begin{enumerate}
\item The data sequence $(t_i,y_i)$ is independent and identically distributed in accordance with the model ~\eqref{eq:Intro:model}, with $\mu^\dagger\in (L^\infty)^k$, $\varphi(i) \sim \mathrm{Cat}(p)$, $\epsilon_i \sim \phi_0$, $t_i\sim \phi_T$: $\varphi(i), \epsilon_i$ and $t_i$ are mutually independent, and $(\varphi(i), \epsilon_i, t_i)$, $(\varphi(j),\epsilon_j,t_j)$ are independent for $i\neq j$. 
We assume $\phi_0$ and $\phi_T$ are continuous densities with respect to the Lebesgue measure on $\mathbb{R}^d$ and $[0,1]$ respectively and use the same symbols to refer to these densities and to their associated measures.
\item The density $\phi_0$ is centered and has finite second moments.
\item For all $\epsilon\in\mathbb{R}^d$, $\phi_0(\epsilon)>0$.
%\item The density $\phi_Y$ satisfies the rate of decay for all $M > M_0$ for some $M_0 < \infty$:  $\int_{\mathbb{R}^d\setminus B(0,M)} |Y| \phi_Y(y|t) \; \mathrm{d} y < L / M$ uniformly in $t$, i.e. there exists $L$ such that for all $t\in [0,1]$.
%\item There exists constants $\tau$ and $\tau_j$ such that for any orthonormal basis $\{e_j\}_{j=1}^d$ of $\mathbb{R}^d$ there exists $\phi_{j,Y}$ such that for all $j=1,2,\dots,d$ 
%\[ \phi_Y(y|t) \leq \prod_{j=1}^d \phi_{j,Y}(y\cdot e_j|t), \quad \quad \|\phi_{j,Y}\|_{L^\infty} < \tau
%\quad \text{and} \quad \sup_{t\in[0,1]} \int_{\mathbb{R}} |y_j| \phi_{j,Y}(y_j|t) \; \mathrm{d} y_j \leq \tau_j. \]
\item There exists $\alpha<-d-3$ and $c_1$ such that $\sup_{t\in [0,1]} \phi_Y(y|t) \leq c_1 |y|^\alpha$. 
%
%
%Let $P$ be a hyperplane in $\mathbb{R}^d$, then there exists a constant $\kappa_P$ (depending on $P$ such that
%\[ \int_{\{\matrm{dist}(y,P)<r\}} |y| \phi_Y(y|t) \; \mathrm{d} y \leq \kappa_P r \quad \quad \forall \, t\in [0,1]. \]
\end{enumerate}
\end{assumptions}

Observe that
\begin{align*}
f_n(\mu^\dagger) & = \frac{1}{n} \sum_{i=1}^n \bigwedge_{j=1}^k |\mu^\dagger_j(t_i) - y_i|^2  + \lambda \sum_{j=1}^k \|\nabla^s \mu_j^\dagger\|_{L^2}^2 \\
 & \leq \frac{1}{n} \sum_{i=1}^n |\mu^\dagger_{\varphi(i)}(t_i) - y_i|^2 + \lambda \sum_{j=1}^k \|\nabla^s \mu_j^\dagger\|_{L^2}^2 \\
 & = \frac{1}{n} \sum_{i=1}^n \epsilon_i^2 + \lambda \sum_{j=1}^k \|\nabla^s \mu_j^\dagger\|_{L^2}^2 \\
 & \to \text{Var}(\epsilon_i) + \lambda \sum_{j=1}^k \|\nabla^s \mu_j^\dagger\|_{L^2}^2 =: \alpha < \infty
\end{align*}
where the convergence is almost surely by the strong law of large numbers.
Hence Assumption~2 implies that there exists $N$ such that $\min_{\mu\in\Theta} f_n(\mu) < \alpha+1$ for $n\geq N$ and $N<\infty$ with probability one (although $N$ could depend on the sequence $\{t_i,y_i\}_{i=1}^n$ and so we could have $\sup_{\omega\in\Omega} N=\infty$). 

To simplify our proofs we use Assumption~3 although the results of this paper can be proved without it.
The assumption is used in bounding the minimizers of $f_n$.
Clearly if $\phi_0$ has bounded support then each $y_i$ is uniformly bounded (a.s.) and one can show that $|\mu^n(t)|$ is bounded uniformly in $n$ and $t$ (a.s.).
Assumption~3 can be relaxed at the expense of some trivial but notationally
messy modifications.

Assumption~4 is used the next section to uniformly control the decay in the density $\phi_Y$.
In particular the assumption allows us bound the error due to restricting to bounded sets.
Although Assumption~4 implies that $\phi_0$ has at least two moments we include the second moment condition in Assumption~2 as the decay in density is not needed until later sections.

Note the second moment condition implies that $\phi_0$ decays as $|\epsilon|\to \infty$ and therefore, by continuity, $\phi_0$ is bounded in $L^\infty$.

We now state the main result for this section.
The proof is an application of Theorem~\ref{thm:Prelim:Gamma:Conmin} once we have shown that $f_\infty$ is the $\Gamma$-limit (Theorem~\ref{thm:Conv:GammaConv}) and established the uniform bound on the set of minimizers Theorem~\ref{thm:Conv:Bound} (which by reflexivity of the space $(H^s)^k$ implies weak compactness).

\begin{theorem}
\label{thm:Conv:Conv}
Define $f_n,f_\infty:\Theta\to\mathbb{R}$ by~\eqref{eq:Intro:fn} and~\eqref{eq:Intro:finfty} respectively, where $\Theta\subset (H^s)^k$ for $s\geq 1$ is given by~\eqref{eq:Conv:Theta}.
Under Assumptions 1-3 any sequence of minimizers $\mu^n$ of $f_n$ is, with probability one, weakly compact and any weak limit $\mu^\infty$ is a minimizer of $f_\infty$.
\end{theorem}

\subsection{The \texorpdfstring{$\Gamma$}{Gamma}-Limit \label{subsec:Conv:Gamma}}

We claim the $\Gamma$-limit of $(f_n)$ is given by~\eqref{eq:Intro:finfty}.

\begin{theorem} \label{thm:Conv:GammaConv}
Define $f_n,f_\infty:\Theta\to\mathbb{R}$ by~\eqref{eq:Intro:fn} and~\eqref{eq:Intro:finfty} respectively where $\Theta\subset (H^s)^k$ for $s\geq 1$ is given by~\eqref{eq:Conv:Theta}.
Under Assumptions 1-2
\[ f_\infty = \Gamma\text{-}\lim_n f_n \]
for almost every sequence of observations $(t_1,y_1),(t_2,y_2),\dots$.
\end{theorem}

\begin{proof}
We are required to show that the two inequalities in Definition~\ref{def:Prelim:Gamma:Gamcon} hold with probability 1.
In order to do this we follow~\cite{thorpe15} and consider a subset of $\Omega$ of full measure, $\Omega^\prime$, and show that both statements hold for every data sequence obtained from that set.

For clarity let $P(\text{d} (t,y)) = \phi_Y(\text{d}y|t) \phi_T(\text{d}t)$.
Let $P_n^{(\omega)}$ be the associated empirical measure arising from the
particular elementary event $\omega$, which we define via it's action on any
continuous bounded function $h:[0,1]\times\mathbb{R}^d\to \mathbb{R}$:  $P_n^{(\omega)} h =
\frac{1}{n} \sum_{i=1}^n h\left(t_i^{(\omega)},y_i^{(\omega)}\right)$ where
$\left(t_i^{(\omega)},y_i^{(\omega)}\right)$ emphasizes that these are the observations
associated with elementary event $\omega$. Define $g_\mu(t,y) = \bigwedge_{j=1}^k (y-\mu_j(t))^2$.
To highlight the dependence of $f_n$ on $\omega$ we write $f_n^{(\omega)}$.
We can write
\[ f_n^{(\omega)}(\mu) = P_n^{(\omega)} g_\mu + \lambda\sum_{j=1}^k\|\nabla^s\mu_j\|_{L^2}^2 \quad \quad \text{and} \quad \quad f_\infty = Pg_\mu + \lambda\sum_{j=1}^k\|\nabla^s\mu_j\|_{L^2}^2. \]
We define
\begin{align*}
\Omega^\prime & = \left\{ \omega\in \Omega : P_n^{(\omega)}\Rightarrow P \right\} \cap  \left\{ \omega\in \Omega : P_n^{(\omega)}(B(0,q)^c) \to P(B(0,q)^c) \; \forall q\in \mathbb{N} \right\} \\
 & \quad \quad \quad \quad \quad \cap \left\{ \omega\in \Omega : \int_{(B(0,q))^c}|y|^2 \; P_n^{(\omega)}(\text{d} (t,y)) \to \int_{(B(0,q))^c} |y|^2 \; P(\text{d}(t,y)) \; \forall q\in \mathbb{N} \right\}
\end{align*}
then $\mathbb{P}(\Omega^\prime)=1$ by the almost sure weak convergence of the empirical measure~\cite{dudley02} and the strong law of large numbers.

Fix $\omega\in \Omega^\prime$ and we start with the lim inf inequality.
Let $\mu^n\rightharpoonup \mu$.
By Theorem~1.1 in~\cite{feinberg14} we have
\[ \int_{[0,1]\times \mathbb{R}^d} \liminf_{n\to \infty, (t^\prime,y^\prime)\to (t,y)} g_{\mu^n}((t^\prime,y^\prime)) \; P(\text{d} (t,y)) \leq \liminf_{n\to \infty} \int_{[0,1]\times \mathbb{R}^d} g_{\mu^n}(t,y) \; P_n^{(\omega)}(\text{d} (t,y)). \]
By the same argument as in Proposition~4.8.ii in~\cite{thorpe15} we have
\[ \liminf_{n\to \infty, (t^\prime,y^\prime)\to (t,y)} \left(y^\prime - \mu^n_j(t^\prime)\right)^2 \geq \left(y - \mu_j(t)\right)^2. \]
Taking the minimum over $j$ we have
\[ \liminf_{n\to \infty, (t^\prime,y^\prime)\to (t,y)} g_{\mu^n}(t^\prime,y^\prime) \geq g_{\mu}(t,y). \]
And, as norms in Banach spaces are weak lower semi-continuous, $\liminf_{n\to \infty} \|\nabla^s\mu_j^n\|_{L^2}^2 \geq \|\nabla^s\mu_j\|_{L^2}^2$.
Therefore
\[ \liminf_{n\to \infty} f_n^{(\omega)}(\mu^n) \geq f_\infty(\mu) \]
as required.

We now establish the existence of a recovery sequence for every $\omega \in \Omega^\prime$ and every $\mu \in \Theta$.
Let $\mu^n=\mu\in \Theta$.
Let $\zeta_q$ be a $C^\infty(\mathbb{R}^{d+1})$ sequence of functions such that $0\leq \zeta_q(t,y) \leq 1$ for all $(t,y)\in \mathbb{R}^{d+1}$, $\zeta_q(t,y) = 1$ for $(t,y)\in B(0,q-1)$ and $\zeta_q(t,y) = 0$ for $(t,y)\not\in B(0,q)$.
Then the function $\zeta_q(t,y)g_\mu(t,y)$ is continuous for all $q$.
We also have, for any $(t,y) \in [0,1]\times\mathbb{R}^d$,
\begin{align*}
\zeta_q(t,y) g_\mu(t,y) & \leq \zeta_q(t,y) |y-\mu_1(t)|^2 \\
 & \leq 2 \zeta_q(t,y) \left(|y|^2 + |\mu_1(t)|^2\right) \\
 & \leq 2\zeta_q(t,y) \left( |y|^2 + \|\mu_1\|_{L^\infty([0,1])}^2 \right) \\
 & \leq 2|q|^2 + 2\|\mu_1\|^2_{L^\infty([0,1])} < \infty
\end{align*}
so $\zeta_q g_\mu$ is a continuous and bounded function, hence by the weak convergence of $P_n^{(\omega)}$ to $P$ we have
\[ P_n^{(\omega)} \zeta_q g_\mu \to P \zeta_q g_\mu \]
as $n\to \infty$ for all $q\in\mathbb{N}$.
For all $q\in\mathbb{N}$ we have
\begin{align*}
\limsup_{n\to \infty} |P_n^{(\omega)} g_\mu - Pg_\mu | & \leq \limsup_{n\to \infty} |P_n^{(\omega)} g_\mu - P_n^{(\omega)}\zeta_q g_\mu | + \limsup_{n\to \infty} |P_n^{(\omega)} \zeta_q g_\mu - P \zeta_q g_\mu | \\
 & \quad \quad \quad \quad \quad \quad + \limsup_{n\to \infty} |P \zeta_q g_\mu - Pg_\mu | \\
 & = \limsup_{n\to \infty} |P_n^{(\omega)} g_\mu - P_n^{(\omega)}\zeta_q g_\mu | + |P \zeta_q g_\mu - Pg_\mu |.
\end{align*}
Therefore,
\[ \limsup_{n\to \infty} |P_n^{(\omega)} g_\mu - Pg_\mu | \leq \limsup_{q\to \infty} \limsup_{n\to \infty} |P_n^{(\omega)} g_\mu - P_n^{(\omega)}\zeta_q g_\mu | \]
by the dominated convergence theorem.
We now show that the right hand side of the above expression is equal to zero.
We have
\begin{align*}
|P_n^{(\omega)} g_\mu - P_n^{(\omega)}\zeta_q g_\mu | & \leq P_n^{(\omega)} \mathbb{I}_{(B(0,q-1))^c} g_\mu \\
 & \leq \int_{[0,1]\times \mathbb{R}^d} \mathbb{I}_{(B(0,q-1))^c}(t,y) |y-\mu_1(t)|^2 \; P_n^{(\omega)}(\text{d}(t,y)) \\
 & \leq 2\int_{[0,1]\times \mathbb{R}^d} \mathbb{I}_{(B(0,q-1))^c}(t,y) |y|^2 \; P_n^{(\omega)}(\text{d}(t,y)) \\
 & \quad \quad \quad + 2\|\mu_1\|_{L^\infty([0,1])}^2 \int_{[0,1]\times \mathbb{R}^d} \mathbb{I}_{(B(0,q-1))^c}(t,y) \; P_n^{(\omega)}(\text{d}(t,y)) \\
 & \to 2\int_{[0,1]\times \mathbb{R}^d} \mathbb{I}_{(B(0,q-1))^c}(t,y) |y|^2 \; P(\text{d}(t,y)) \\
 & \quad \quad \quad + 2\|\mu_1\|_{L^\infty([0,1])}^2 \int_{[0,1]\times \mathbb{R}^d} \mathbb{I}_{(B(0,q-1))^c}(t,y) \; P(\text{d}(t,y)) \quad \text{as } n\to \infty \\
 & \to 0 \quad \text{as } q \to \infty
\end{align*}
where the last limit follows by the monotone convergence theorem and Assumption~2.
We have shown
\[ \lim_{n\to \infty} |P_n^{(\omega)} g_\mu - Pg_\mu | = 0. \]
Hence
\[ f_n^{(\omega)}(\mu) \to f_\infty(\mu) \]
as required.
\end{proof}

\subsection{Boundedness \label{subsec:Conv:Bound}}

The aim of this subsection is to show that the minimizers of $f_n$ are uniformly bounded in $n$ for almost every sequence of observations.
We divide this into two parts; bounding each of the $\mathcal{H}_0$ and $\mathcal{H}_1$ norms.
The $\mathcal{H}_1$ bound follows easily from the regularization.
For the $\mathcal{H}_0$ bound we exploit the equivalence of norms on finite-dimensional vector spaces to choose a convenient norm on  $\mathcal{H}_0$.

By the argument which followed the assumptions we have, for $n$ sufficiently large and with probability one, $\min_{\mu\in \Theta} f_n(\mu) \leq \alpha + 1 <\infty$.
Now we let $\mu^n$ be a sequence of minimizers.
Then there exists $\hat{\Omega}\subset\Omega$ such that $\mathbb{P}(\hat{\Omega})=1$ and for all $\omega\in \hat{\Omega}$ we have
\[ f_n(\mu^\dagger) = P_n^{(\omega)}g_{\mu^\dagger} + \lambda \sum_{j=1}^k \|\nabla^s\mu^\dagger_j\|_{L^2}^2 \to Pg_{\mu^\dagger} + \lambda \sum_{j=1}^k \|\nabla^s\mu^\dagger_j\|_{L^2}^2 =:\alpha. \]
Therefore for all $\omega\in\hat{\Omega}$ there exists $N=N(\omega)$ such that for $n\geq N$ we have
\[ \lambda\sum_{j=1}^k \|\mu^n_j\|_1^2 \leq f_n(\mu^n) \leq f_n(\mu^\dagger) \leq \alpha + 1. \]
Therefore $\|\mu^n_j\|_1$ is bounded almost surely for each $j$.
We are left to show the corresponding result for $\|\mu^n_j\|_0$.

The following lemma will be used to establish the main result of this subsection, Theorem~\ref{thm:Conv:Bound}.
It shows that, if for some sequence $\nu^n\in H^s$ with $\|\nabla^s \nu^n\|_{L^2} \leq \sqrt{\alpha}$ and $\|\nu^n\|_0\to \infty$, then we have that, up to a subsequence, $|\nu^n(t)|\to \infty$ with the exception of at most finitely many $t\in [0,1]$.
When applied to $\mu^n_j$ this will be used to show that in the limit, if any center is unbounded, then the minimization can be achieved over $k-1$ clusters --- and hence to provide a contradiction.

\begin{lemma}
\label{lem:Conv:Bound:UnifInfty}
Let $\nu \in H^s$ satisfy $\|\nabla^s\nu^n\|_{L^2} \leq \sqrt{\alpha}$ and $\|\nu^n\|_0\to \infty$.
Then there exists a subsequence such that, with the exception of at most finitely many $t\in [0,1]$, we have $|\nu^{n_m}(t)|\to \infty$.
Furthermore for each $t\in (0,1)$ with $|\nu^n(t)|\to \infty$ and any $t_n\to t$ we have $|\nu^n(t_n)|\to \infty$.
\end{lemma}

\begin{proof}
Let the norm on $\mathcal{H}_0$ be given by
\begin{equation}\label{eq:h0norm}
 \|\nu \|_0 := \sum_{i=0}^{s-1} \frac{|\nabla^i \nu(0)|}{i!}. 
\end{equation}
By Taylor's theorem and the bound on $\|\nabla^s \nu^n\|_{L^2}$ we have
\[ \left| \nu^n(t) - \sum_{i=0}^{s-1} \frac{\nabla^i \nu^n(0)}{i!} t^i \right| \leq \sqrt{\alpha}. \]
Now let $Q_n(t) = \sum_{i=0}^{s-1} \frac{\nabla^i \nu^n(0)}{i!} t^i$ and $\hat{Q}_n(t) = \frac{Q_n(t)}{\|Q_n\|_0}$.
In particular $\|\hat{Q}_n\|_0=1$.
Take any subsequence $n_m$ then since $\frac{\mathrm{d}^i \hat{Q}_n}{\mathrm{d} t^i}$ are uniformly bounded equi-continuous for all $i=0,1,\dots, s-1$ so by the Arzel\`{a}-Ascoli theorem there exists a further subsequence (which we relabel) for which $\frac{\mathrm{d}^i\hat{Q}_n}{\mathrm{d} t^i}$ converges uniformly to $\frac{\mathrm{d}^i \hat{Q}}{\mathrm{d} t^i}$ for some $\hat{Q}$ and all $i=0,1,\dots s-1$.
In particular $\frac{\mathrm{d}^{s-1} \hat{Q}}{\mathrm{d} t^{s-1}}$ is a constant and therefore $\hat{Q}$ is a polynomial of degree at most $s-1$.
It follows that $\|\hat{Q}\|_0=1$ and therefore $\hat{Q}$ is not identically zero, hence $\hat{Q}$ has at most $s-1$ roots.
For any $t$ that is not a root of $\hat{Q}$ we have $|Q_{n_m}(t)| =
|\hat{Q}_{n_m}(t)| \|Q_{n_m}\|_0 \to \infty$.
This implies that $|\nu^n(t)|\to \infty$.

%Now if we assume that there is a subsequence of $Q_n$ that is bounded for infinitely many $t$ then by applying the above we can find a subsequence that diverges at all but finitely many points contradicting the assumption of such a subsequence.
%Hence all subsequences of $Q_n(t)$ diverge except for at most finitely many points.
%As there are at most countably many subsequences then it follows that $Q_n$ diverges except at most countable many points.

Now pick $t\in [0,1]$ with $|\nu^n(t)|\to \infty$ and assume $t_n\to t$.
We assume that there exists a subsequence $n_m$ such that $|Q_{n_m}(t_{n_m})|$ is bounded.
By going to a further subsequence (which we relabel) we assume that $\hat{Q}_{n_m} \to \hat{Q}$ uniformly.
Choose $\delta>0$ sufficiently small then there exists $\epsilon>0$ and $N<\infty$ such that for all $s$ with $|s-t|<\epsilon$ and $n_m\geq N$ then
\[ |\hat{Q}(s)|\geq \delta, \quad \quad \|\hat{Q}_{n_m}-\hat{Q}\|_{L^\infty} \leq \frac{\delta}{2} \quad \quad \text{and} \quad \quad |t_{n_m}-t|\leq \epsilon. \]
It follows that
\[ |\hat{Q}_{n_m}(t_{n_m})| \geq |\hat{Q}(t_{n_m})| - |\hat{Q}(t_{n_m}) - \hat{Q}_{n_m}(t_{n_m})| \geq \frac{\delta}{2}. \]
In particular $|Q_{n_m}(t_{n_m})| = \|Q_n\|_0 |\hat{Q}_{n_m}(t_{n_m})| \geq \frac{\delta \|Q_{n_m}\|_0}{2} \to \infty$.
This contradicts the assumption that $|Q_{n_m}(t_{n_m})|$ is bounded.
We have shown that $|\nu^n(t_n)|\to \infty$.
\end{proof}

We proceed to the main result of this subsection.

\begin{theorem}
\label{thm:Conv:Bound}
Define $f_n,f_\infty:\Theta\to\mathbb{R}$, where $\Theta\subset (H^s)^k$ for $s\geq 1$ is given by~\eqref{eq:Conv:Theta}, by~\eqref{eq:Intro:fn} and~\eqref{eq:Intro:finfty}
respectively.
Let $\mu^n$ be a minimizer of $f_n$ then, under Assumptions 1-3, for almost every sequence of observations there exists a constant $M < \infty$ such that $\|\mu^n\|_{(H^s)^k}\leq M$ for all $n$.  
\end{theorem}

\begin{proof}
As in the proof of Theorem~\ref{thm:Conv:GammaConv} we let $\omega\in \Omega^{\prime\prime}$ where
\begin{align*}
\Omega^{\prime\prime} & = \left\{\omega\in\Omega^\prime: \frac{1}{n} \sum_{i=1}^n \epsilon_i^2 \to \text{Var}(\epsilon_1) \right\} \\
 & \quad \quad \quad \quad \bigcap \left( \cap_{c\in \mathbb{Q}^d} \left\{ \omega\in \Omega^\prime: P_n^{(\omega)}\left( B\left(c,\frac{\delta}{4}\right) \right) \rightarrow P\left( B\left(c,\frac{\delta}{4}\right)\right) \right\} \right)
\end{align*}
where $\Omega^\prime$ is defined in the proof of Theorem~\ref{thm:Conv:GammaConv}.
We have $\mathbb{P}(\Omega^{\prime\prime})=1$.
For the remainder of the proof we assume $\omega\in \Omega^{\prime\prime}$.
Then there exists $N^{(\omega)}<\infty$ such that $f_n^{(\omega)}(\mu^n)\leq \alpha + 1$ for all $n\geq N^{(\omega)}$.
Hence, for sufficiently large $n$,
\[ \lambda\sum_{j=1}^k \|\mu^n_j\|_1^2 \leq f_n^{(\omega)}(\mu^n)\leq \alpha +1. \]

It remains to show the $\mathcal{H}_0$ bound.
The structure of the proof is similar to~\cite[Lemma 2.1]{lember03}.
We will argue by contradiction.
In particular we argue that if a cluster center is unbounded then in the limit the minimum is achieved over the remaining $k-1$ cluster centers.

\paragraph*{Step 1:}
\textit{The minimization is achieved over $k-1$ cluster centers.}
We assume $\sup_j \|\mu_j^n\|_0$ is unbounded, then there exists $j^*$ and a subsequence (which we relabel) such that $\|\mu^n_{j^*}\|_0\to \infty$.
By Lemma~\ref{lem:Conv:Bound:UnifInfty} there exists a further subsequence (again relabelled) such that $|\mu^n_{j^*}(t)|\to \infty$ for all but finitely many $t$.
For any such $t$, by Lemma~\ref{lem:Conv:Bound:UnifInfty}, we have
%Assume there exists $j^*$ such that $\|\mu^n_{j^*}\|_0\to \infty$.
%Pick $t$ such that $|\mu^n_{j^*}(t)|\to\infty$.
%By Lemma~\ref{lem:Conv:Bound:UnifInfty} we have
\[ \lim_{n\to \infty,t'\to t} |\mu^n_{j^*}(t')| = \infty. \]
This easily implies
\[ \lim_{n\to \infty, (t^\prime,y^\prime)\to (t,y)} \left| \mu^n_{j^*}(t^\prime) - y^\prime\right|^2 = \infty \]
for any $y\in\mathbb{R}^d$.
Therefore
\[ \liminf_{n\to \infty, (t^\prime,y^\prime)\to (t,y)} \left( \bigwedge_{j=1}^k \left|\mu_j^n(t^\prime) - y^\prime\right|^2 - \bigwedge_{j\neq j^*} \left|\mu_j^n(t^\prime) - y^\prime \right|^2 \right) = 0. \]
Note that the above expression holds for $P$-almost every $(t,y)\in
[0,1]\times\mathbb{R}^d$ (as by Lemma \ref{lem:Conv:Bound:UnifInfty} the collection
  of $t$ for which $|\mu^n_{j^*}(t)|\not\to\infty$ has Lebesgue measure zero).
By Fatou's lemma for weakly converging measures~\cite[Theorem 1.1]{feinberg14} and the above we have
\[ \liminf_{n\to \infty} \left( \int_{[0,1]\times \mathbb{R}^d} \bigwedge_{j=1}^k |\mu_j^n(t)-y|^2 - \bigwedge_{j\neq j^*} |\mu_j^n(t)-y|^2 \; P_n^{(\omega)}(\text{d}t,\text{d}y) \right) \geq 0. \]
Hence 
\[ \liminf_{n\to \infty} \left( f_n^{(\omega)}(\mu^n) - f_n^{(\omega)}((\mu_j^n)_{j\neq j^*}) - \lambda \| \nabla^s \mu_{j^*}^n\|^2_{L^2} \right) \geq 0 \]
where we interpret $f_n^{(\omega)}((\mu_j^n)_{j\neq j^*})$ accordingly.
So,
\[   \liminf_{n\to \infty} \left( f_n^{(\omega)}(\mu^n) - f_n^{(\omega)}((\mu_j^n)_{j\neq j^*}) \right)
\geq 0 . \]

\paragraph*{Step 2:}
\textit{The contradiction.}
If we can show that there exists $\epsilon>0$ such that  
\[ \liminf_{n\to \infty} \left( f_n^{(\omega)}(\mu^n) - f_n^{(\omega)}((\mu_j^n)_{j\neq j^*}) \right) \leq -\epsilon. \]
(i.e. we can do strictly better by fitting $k$ centers than fitting $k-1$ centers) then we can conclude the theorem.

Now,
\[ f_n^{(\omega)}(\mu^n) \leq f_n^{(\omega)}(\hat{\mu}^n) = \frac{1}{n} \sum_{i=1}^n \bigwedge_{j=1}^k |\hat{\mu}^n_j(t_i) - y_i|^2 + \lambda \sum_{j\neq j^*} \|\nabla^s\hat{\mu}_j^n\|_{L^2}^2, \]
where
\[ \hat{\mu}_j^n(t) = \left\{ \begin{array}{ll} \mu^n_j(t) & \text{for } j\neq j^* \\ c_n & \text{for } j=j^* \end{array} \right. \]
for a constant $c_n$.
By definition, the $\hat{\mu}_j^n$ must have a minimum separation distance of $\delta$.
For now we assume that we can choose $c_n$ such that this criterion is fulfilled.
So if $|y_i-c_n|\leq \frac{\delta}{4}$ then
\[ |y_i-c_n| + \frac{\delta}{4} \leq |\mu_j^n(t_i)-y_i| \]
for all $j\neq j^*$.
And therefore $|y_i-c_n|^2 + \frac{\delta^2}{16} \leq |\mu_j^n(t_i)-y_i|^2$ which implies
\begin{align*}
f_n^{(\omega)}((\mu_j^n)_{j\neq j^*}) & = \frac{1}{n} \sum_{i=1}^n \bigwedge_{j\neq j^*} |\mu_j^n(t_i) - y_i|^2 + \lambda \sum_{j\neq j^*} \|\nabla^s \mu_j\|_{L^2}^2 \\
 & = \frac{1}{n} \sum_{i=1}^n \bigwedge_{j\neq j^*} |\mu_j^n(t_i) - y_i|^2 \mathbb{I}_{(t_i,y_i)\nsim_n j^*} + \frac{1}{n} \sum_{i=1}^n \bigwedge_{j\neq j^*} |\mu_j^n(t_i) - y_i|^2 \mathbb{I}_{(t_i,y_i)\sim_n j^*} \\
 & \quad \quad \quad + \lambda \sum_{j\neq j^*} \|\nabla^s \mu_j\|_{L^2}^2 \\
 & \geq \frac{1}{n} \sum_{i=1}^n \bigwedge_{j\neq j^*} |\mu_j^n(t_i) - y_i|^2 \mathbb{I}_{(t_i,y_i)\nsim_n j^*} + \frac{1}{n} \sum_{i=1}^n |c_n - y_i|^2 \mathbb{I}_{(t_i,y_i)\sim_n j^*} \\
 & \quad \quad \quad + \frac{\delta^2}{16} P_n^{(\omega)}\left( [0,1]\times B\left(c_n,\frac{\delta}{4}\right) \right) + \lambda \sum_{j\neq j^*} \|\nabla^s \mu_j\|_{L^2}^2 \\
 & = f_n^{(\omega)}(\hat{\mu}^n) + \frac{\delta^2}{16} P_n^{(\omega)}\left( [0,1]\times B\left(c_n,\frac{\delta}{4}\right) \right).
\end{align*}
Where $(t_i,y_i)\sim_n j$ means coordinate $(t_i,y_i)$ is associated to center $\hat{\mu}^n_j$ in the sense that $(t,y) \sim_n j \Leftrightarrow j = \argmin_{i=1,\ldots,k} |y - \hat{\mu}_i^n(t)|$ (and if the minimum is not uniquely achieved then we take the smallest $j$ such that $j\in \argmin_{i=1,\dots,k}|y-\hat{\mu}_i^n(t)|$).
If we can show that $P_n^{(\omega)}\left( [0,1]\times B\left(c_n,\frac{\delta}{4}\right) \right)$ is bounded away from zero, then the result follows.

Since we assumed $\epsilon_1$ has unbounded support on $\mathbb{R}^d$ if we can show
that $|c_n|\leq M$ for a constant $M$ and $n$ sufficiently large (a.s.) then we
can infer the existence of a subsequence such that
\[ \liminf_{n\to \infty} P_n^{(\omega)}\left( [0,1] \times B\left(c_n,\frac{\delta}{4}\right) \right) = \lim_{m\to \infty} P_{n_m}^{(\omega)}\left( [0,1]\times B\left(c_{n_m},\frac{\delta}{4}\right) \right) \]
and $c_{n_m}$ converges to some $c$.
This implies (after applying Fatou's lemma for weakly converging measures~\cite[Theorem 1.1]{feinberg14})
\begin{align*}
\liminf_{n\to \infty} P_n^{(\omega)}\left( [0,1]\times B\left(c_n,\frac{\delta}{4}\right) \right) &
\geq \lim_{m\to \infty} P_{n_m}^{(\omega)}\left( [0,1]\times B\left(c_{n_m},\frac{\delta}{4}\right)\right) \\
 & \geq P\left([0,1]\times B\left(c,\frac{\delta}{4}\right)\right) \\
 & = \int_0^1 \int_{\mathbb{R}^d} \mathbb{I}_{|y-c|\leq \frac{\delta}{4}} \phi_Y(y|t) \phi_T(t) \; \text{d} y \text{d} t.
\end{align*}
By Assumption 3 and the continuity in Assumption~1, 
there exists $\epsilon'>0$ such that $ \phi_Y(y|t)\geq \epsilon'$ for all $y\in [-M,M]^d$ and $t\in [0,1]$.
Hence we may bound the final expression above by
\[ \inf_{c\in [-M,M]} \int_0^1 \int_{\mathbb{R}^d} \mathbb{I}_{|y-c|\leq \frac{\delta}{4}} \phi_Y(y|t) \phi_T(t) \; \text{d} y \text{d} t \geq \epsilon' \mathrm{Vol}\left(B\left(0,\frac{\delta}{4}\right) \right). \]

We are left to show such an $M$ exists.
Assume there exists $M_{k-1}$ such that for all $j\neq j^*$ we have $\|\mu^n_j\|_{H^s} \leq M_{k-1}$.
By the Sobolev embedding of $H^s$ into $L^\infty$ there exists a constant $C'$ such that $\|\mu\|_{L^\infty} \leq C' \|\mu\|_{H^s}$ for all $\mu \in H^s$.
%Hence we can bound the size of the range of each cluster center $\mu_j$ by
%\begin{align*}
%\max_{j\neq j^*} \left( \sup_{t\in [0,1]} \mu^n_j(t) - \inf_{t\in [0,1]} \mu^n_j(t) \right) & \leq 2 \max_{j\neq j^*} \|\mu_j^n\|_{L^\infty} \\
% & \leq 2C' \max_{j\neq j^*} \|\mu^n\|_{H^s} \\
% & \leq 2 C' M_{k-1}.
%\end{align*}
%In particular $k-1$ centers have size of range at most $2C'M_{k-1}(k-1) + \delta(k-2)=:C$.
%Hence there exists $c_n\in [0,C+\delta]^d$ such that $\hat{\mu}_{j^*}^n(t)=c_n$ and $\hat{\mu}^n\in\Theta$.
And therefore $|\mu^n_j(t)|\leq C'M_{k-1}$ for all $j\neq j^*$ and $t\in [0,1]$.
Let $C=C'M_{k-1}+\delta$ then it follows that there exists $c_n\in [0,C]^d$ such that $\hat{\mu}_{j^*}^n(t) = c_n$ and $\hat{\mu}^n\in \Theta$.

Now if no such $M_{k-1}$ exists then there exists a second cluster such that $\|\mu^n_{j^{**}}\|_{H^s} \to \infty$ where $j^{**}\neq j^*$.
By the same argument
\begin{align*}
\liminf_{n\to \infty} \left( f_n^{(\omega)}(\mu^n) - f_n^{(\omega)}((\mu_j^n)_{j\neq j^*,j^{**}}) \right) & \geq 0 \\
f_n^{(\omega)}(\mu^n) - f_n^{(\omega)}((\mu_j^n)_{j\neq j^*,j^{**}}) & \leq -\frac{\delta^2}{16} P_n^{(\omega)}
\left(B\left(c_n,\frac{\delta}{4}\right)\right) - \frac{\delta^2}{16} P_n^{(\omega)}
\left(B\left(c^\prime_n,\frac{\delta}{4}\right) \right)
\end{align*}
for a constant $c^\prime_n$.
By induction it is clear that we can find $M_{k-l}$ such that $k-l$ cluster centers are bounded.
The result then follows.
\end{proof}

\begin{remark}
\label{rem:Conv:Bound:choiceofk}
Note that in the above theorem we did not need to assume a correct choice of $k$.
If the true number of cluster centers is $k'$ and we incorrectly use $k\neq k'$, then the resulting cluster centers are still bounded.
In fact for all the results of this paper the correct choice of $k$ is not necessary: although the minimizers of $f_\infty$ may no longer make physical sense, the problem is still robust in that the conclusions of Theorems~\ref{thm:Conv:Conv} and~\ref{thm:Rate:StrongConv} and Corollary~\ref{cor:Rate:Rate} hold. 
\end{remark}

\section{Weak to Strong Convergence \label{sec:Rate}}

We now strengthen the result of the previous section and show that in fact (upto subsequences) convergence of minimizers is strong in $H^s$.
Our proof is based on the methodology Pollard used for proving the central limit theorem for the $k$-means method in Euclidean spaces~\cite{pollard82}.
In Pollard's proof he assumed a positive definiteness condition on the second derivative of, what we call in this paper, $f_\infty$.
Under an analogous condition we are also able to give a rate of convergence on convergent sequences of minimizers.
%The authors know of no non-restrictive assumptions on the distributions $\phi_0$ and $\phi_T$ that will imply this condition.
Whether this condition holds will depend on the interplay between the integral over the boundaries of each partition and the size of each partition.
%In particular a careful calculation
%\todo{Probably need to give details here, if only for the referee. This feels a bit ugly\dots - I agree it is ugly and since I don't think its a good characterisation of when $\partial_-^2f_\infty$ is positive definite then I don't think its worth including it. The lower bound that I had in mind for $\kappa$ was such a crude bound that the result would have had very little power.}
%of $\partial^2_-f_\infty(\mu;\nu)$ shows that
%\[ \partial_-^2 f_\infty (\mu;\nu) \geq \kappa \|\nu\|_{(L^2)^k}^2 + 2 \lambda \|\nabla \nu\|_{(L^2)^k}^2 - B(\mu,\nu) \]
%where $B(\mu,\nu)$ is the integral over neighbourhoods of the boundaries of each partition defined by $\mu$,
%\begin{align*}
%B(\mu,\nu) = & 2 \limsup_{r\to 0} \frac{1}{r} \int_0^1 \int_{y:j_r(t,y)\neq j(t,y)} \Big\{ \left( y - \mu_{j_r(t,y)}(t) \right) \cdot \nu_{j_r(t,y)}(t) \\
% & \quad \quad \quad - \left(y - \mu_{j(t,y)}(t)\right) \cdot \nu_{j(t,y)}(t) \Big\} \phi_Y(y|t) \; \phi_T(t) \; \mathrm{d} y \; \mathrm{d} t
%\end{align*}
%and
%\begin{equation} \label{eq:Rate:jr}
%j_r(t,y) = \argmin_j |y-\mu_j(t)-r \nu_j(t)|, \quad \quad \quad j(t,y) = j_0(t,y).
%\end{equation}
%The constant $\kappa$ can be shown to be strictly positive if $\phi_0(t)$ is bounded away from 0.
%Whether $\partial^2_-f_\infty$ is positive definite at $\mu$ then depends on whether $\inf_\nu \left(\kappa - \frac{B(\mu,\nu)}{\|\nu\|_{(L^2)^k}^2} \right)> 0$.

We state the main results of this section now but leave the proofs to the end.

\begin{theorem}
\label{thm:Rate:StrongConv}
Define $f_n,f_\infty:\Theta\to\mathbb{R}$, where $\Theta$ is given by~\eqref{eq:Conv:Theta}, by~\eqref{eq:Intro:fn} and~\eqref{eq:Intro:finfty}, respectively.
Let $\{\mu^n\}_{n\in\mathbb{N}}\subset \Theta$ where $\mu^n$ minimizes $f_n$.
Let $\mu^{n_m}$ be any subsequence that weakly converges almost surely to some $\mu^{\infty}$ then under Assumptions~1-4 we have that, after passing to a further subsequence, $\mu^{n_m}$ converges to $\mu^\infty$ strongly in $H^s$ and in probability.
\end{theorem}

\begin{corollary}
\label{cor:Rate:Rate}
If in addition to the conditions in Theorem~\ref{thm:Rate:StrongConv} and where $\mu^\infty$ is a minimizer of $f_\infty$ we assume that there exists $\rho>0$ and $\kappa>0$ such that
\[ \partial^2_- f_\infty(\mu;\nu) \geq \kappa \|\nu\|_{(H^s)^k}^2 \]
for all $\mu$ with $\|\mu-\mu^\infty\|_{(H^s)^k}\leq \rho$.
Then any sequence $\mu^n$ of minimizers with $\mu^n\to \mu^\infty$ in $H^s$ obeys the rate of convergence
\[ \|\mu^n - \mu^\infty\|_{(H^s)^k}^2 = O_p\left(\frac{1}{n}\right). \] 
\end{corollary}

For clarity, we will assume that the entire sequence $\mu^n$ weakly converges in the remainder of this paper to avoid reference to subsequences.
Relaxing this assumption is trivial, but notationally cumbersome.

We let $Y_n(\mu)=\sqrt{n}(f_n(\mu)-f_\infty(\mu))$ and then, by Taylor expanding around $\mu^\infty$, we have
\[ Y_n(\mu^n) = Y_n(\mu^\infty) + \partial Y_n(\mu^\infty;\mu^n-\mu^\infty) + \text{h.o.t.} \]
In Lemma~\ref{lem:Rate:d2Yn}, using Chebyshev's inequality, we bound the G\^ateaux derivative of $Y_n$ in probability.
Similarly one can Taylor expand $f_\infty$ around $\mu^\infty$.
After some manipulation of the Taylor expansion, where we leave the details until the proof of Theorem~\ref{thm:Rate:StrongConv}, one has
\[ \partial^2_- f_\infty(\mu^\infty;\mu^n-\mu^\infty) \leq f_n(\mu^n)-f_n(\mu^\infty) + O_p\left(\frac{1}{\sqrt{n}} \|\mu^n-\mu^\infty\|_{(L^2)^k}\right). \]
%where we leave the details until the proof of Theorem~\ref{thm:Rate:StrongConv} at the end of the section.
We note that $f_n(\mu^n)-f_n(\mu^\infty)\leq 0$.
We also show that $2\lambda\|\nabla^s \nu\|_{(L^2)^k}^2 - 2\|\nu\|_{(L^\infty)^k}^2 \leq \partial^2_- f_\infty(\mu^\infty;\nu)$.
Therefore
\[ \lambda \|\nabla^s\left(\mu^n-\mu^\infty\right)\|_{(L^2)^k}^2 \leq  O_p\left(\frac{1}{\sqrt{n}} \|\mu^n-\mu^\infty\|_{(L^2)^k} + \|\mu^n-\mu^\infty\|_{(L^\infty)^k}^2\right). \]
The above expression allows us to convert weak convergence into strong convergence.
Lemmata \ref{lem:Rate:GatDerFinfty} and \ref{lem:Rate:Gat2ndDerFinfty} provide
the first G\^ateaux derivative and a lower bound on the second G\^ateaux
derivatives of $f_\infty$, respectively.

\begin{lemma} \label{lem:Rate:GatDerFinfty}
Define $f_\infty$ by~\eqref{eq:Intro:finfty} and $\Theta\subset (H^s)^k$ for $s\geq 1$ by~\eqref{eq:Conv:Theta}.
Then, under Assumptions~1, 2 and~4, for $\mu\in \Theta \cap (L^\infty)^k$, $\nu\in (H^s)^k$ we have that $f_\infty$ is G\^ateaux differentiable at $\mu$ in the direction $\nu$ with
\begin{align*}
\partial f_\infty(\mu;\nu) = & 2\int_0^1 \int_{\mathbb{R}^d} \left(\mu_{j(t,y)}(t) - y\right) \cdot \nu_{j(t,y)}(t) \phi_Y(y|t) \phi_T(t) \; \text{d} y \text{d} t \\
 & + 2 \lambda \sum_{j=1}^k (\nabla^s \nu_j,\nabla^s \mu_j)
\end{align*}
where $j(t,y)$ is chosen arbitrarily from the set $\argmin_j |y-\mu_j(t)|$, so that
\begin{equation} \label{eq:Rate:jty}
j(t,y) \in 
\argmin_j |y-\mu_j(t)|.
\end{equation}
\end{lemma}

\begin{remark}
\label{rem:Rate:j}
Since $\mu_j$ are continuous the boundary between each element of the resulting partition is itself continuous and has Lebesgue measure zero.
The set on which $j(t,y)$ is not uniquely defined therefore has measure zero.
Hence we will treat $j(t,y)$ as though it was uniquely defined.
\end{remark}

\begin{proof}[Proof of Lemma~\ref{lem:Rate:GatDerFinfty}:]
Fix $\mu\in\Theta$, $\nu\in (H^s)^k$ and $r>0$.
%Recall that $j_r(t,y)$ and $j(t,y)$ are defined by~\eqref{eq:Rate:jr}.
We will assume $d\geq 2$.
The case when $d=1$ simplifies as the boundaries between partitions are points and so we exclude the argument.
Let $\beta=-\frac{1+\epsilon}{\alpha+2+d}$ where $\epsilon>0$ is chosen sufficiently small so that $1-\beta = \frac{\alpha+3+d+\epsilon}{\alpha+2+d}>0$ (true for any $\epsilon< -(\alpha + d + 3)$).
Then
\begin{align}
\frac{1}{r} \int_{|y|\geq r^{-\beta}} |y|^2 \phi_Y(y|t) \, \mathrm{d} y & \leq \frac{c_1}{r} \int_{|y|\geq r^{-\beta}} |y|^{2+\alpha} \, \mathrm{d} y \notag \\
 & = \frac{c}{r} \int_{r^{-\beta}}^\infty t^{2+\alpha+d-1} \, \mathrm{d} t \quad \text{for some } c>0 \notag \\
 & = -\frac{c}{\alpha+2+d} r^{-\beta(\alpha+2+d)-1}. \label{eq:Rate:ResToBall}
\end{align}
Since $\alpha+2+d < 0$ and $-\beta(\alpha+2+d)-1=\epsilon>0$ the above converges to zero as $r\to 0$.
Analogously, one can show $\frac{1}{r} \int_{|y|\geq r^{-\beta}} \, \phi(y|t) \, \mathrm{d} y \to 0$ as $r\to 0$.

Define $j_r(t,y)$ by
\[ j_r(t,y) = \argmin_j |y-\mu_j(t)-r \nu_j(t)|. \]
Then for $(t,y)$ in the interior of the partition associated with $\mu_j$ we have 
\[ j_r(t,y) = j(t,y) \quad \quad \text{for } r \text{ sufficiently small.} \]
More precisely consider two points $y_1,y_2 \in \mathbb{R}^d$, with $|y_1-y_2|\geq \delta$ and let $B_{y_1,y_2}$ be the boundary defined by
\[ B_{y_1,y_2} = \left\{ y \in \overline{B(0,M)} \, : \, |y-y_1| = |y-y_2| \right\} \]
for a constant $M>0$.
%We first consider the case where $d\geq 2$ (where we have to take into account rotations) and 
Let $\tilde{y}_1\in B(y_1,Cr)$ and $\tilde{y}_2 \in B(y_2,Cr)$.
We will denote by $d_H$ the Hausdorff distance between sets in $\mathbb{R}^d$, in particular we wish to bound $d_H(B_{y_1,y_2},B_{\tilde{y}_1,\tilde{y}_2})$.
Elementary geometry implies that this can be bounded by the Euclidean distance between points on the boundary of each set, in particular
\[ d_H(B_{y_1,y_2},B_{\tilde{y}_1,\tilde{y}_2}) \leq d_H(\partial B_{y_1,y_2},\partial B_{\tilde{y}_1,\tilde{y}_2}) \]
where
\[ \partial B_{y_1,y_2} = \left\{ y \in \mathbb{R}^d \, : \, |y|=M \text{ and } |y-y_1| = |y-y_2| \right\}. \]
Without loss of generality assume that $B_{y_1,y_2} \subset \{ x \, : \, x_1=0\}$.
(All assumptions other than 4 are rotation and translation invariant, whilst 4 is rotation invariant it is not translation invariant as the constant $c_1$ could increase with the size of the translation.
However the cluster centers are bounded in $L^\infty$, so in particular the size of the translation can be bounded.
Therefore, up to redefining the constant $c_1$, all the assumptions hold in the rotated and translated coordinate system.
For $d\geq 3$ we consider a cross section at $x_{3:d}=a \in \mathbb{R}^{d-2}$, then there exists constants $\gamma_1,\gamma_2 \in \mathbb{R}$ (depending on $a$) such that $x_1=\gamma_1 x_2 + \gamma_2$ parametrizes the set $\{ x \in B_{\tilde{y}_1,\tilde{y}_2} \, : \, x_{3:d} = a\}$ (for $a>M$ the set is empty and we have nothing to prove).
Let $\theta_a = |\tan^{-1}\gamma_1 | \in [0, \frac{\pi}{2}]$ be the angle between the lines $x_1=0$ and $x_1=\gamma_1 x_2 + \gamma_2$.
When $d=2$ the set $B_{\tilde{y}_1,\tilde{y}_2}$ is already a straight line in $\mathbb{R}^2$ and it is unnecessary to take a cross section (i.e. $x_{3:d}$ is null and $\theta_a$ is independent of $a$).
We will find $\theta^*$ such that $\sin\theta^*=O(r)$ and $\sup_a \theta_a \leq \theta^*$ then we can  bound the Hausdorff distance by
\[ d_H(\partial B_{y_1,y_2},\partial B_{\tilde{y}_1,\tilde{y}_2}) \leq rC + 2M \sin \theta^* = O(r), \]
the above bound holding as it is the maximum distance that can arise from rotation plus the maximum 
possible translation of the set $\partial B_{y_1,y_2}$.
%the maximum rotation plus the maximum translation of the set $\partial B_{y_1,y_2}$.

Let $\ell$ be the ray through $y_1$ and $y_2$ and $\tilde{\ell}$ be the ray through $\tilde{y}_1$ and $\tilde{y}_2$.
Let $P$ be the point of intersection between $\ell$ and $\tilde{\ell}$.
%We assume the point $P$ exists, else $\theta=0$ which is clearly not a maximum.
The point $P$ exists if and only if the lines $\ell$ and $\tilde{\ell}$ are not parallel.
The lines $\ell$ and $\tilde{\ell}$ are parallel if and only if $\theta=0$, trivially any choice of $\theta^*\geq 0$ will bound this case.
Therefore we assume that $\theta>0$ and therefore the point $P$ exists.
 
One can easily show that $\widehat{\tilde{y}_1 P y_1} = \theta$ (the angle between the lines $\tilde{y}_1P$ and $P y_1$ is $\theta$).
There are two possibilities, either (1) $P$ is between $y_1$ and $y_2$ or (2) it isn't.

In the second case we assume that $|y_2-P|< |y_1-P|$ and therefore $|y_1-P|> \delta$.
Let $Q$ be the closest point on $\tilde{\ell}$ to $y_1$ (see Figure~\ref{fig:Rate:2DGeom}).
So, $P,y_1,Q$ form a triangle with $\widehat{PQy_1}=\frac{\pi}{2}$, $\widehat{QPy_1} = \theta$ and $|Q-y_1| \leq |y_1-\tilde{y}_1| \leq Cr$.
Hence $\sin\theta = \frac{|Q-y_1|}{|y_1-P|} \leq \frac{Cr}{\delta}$.

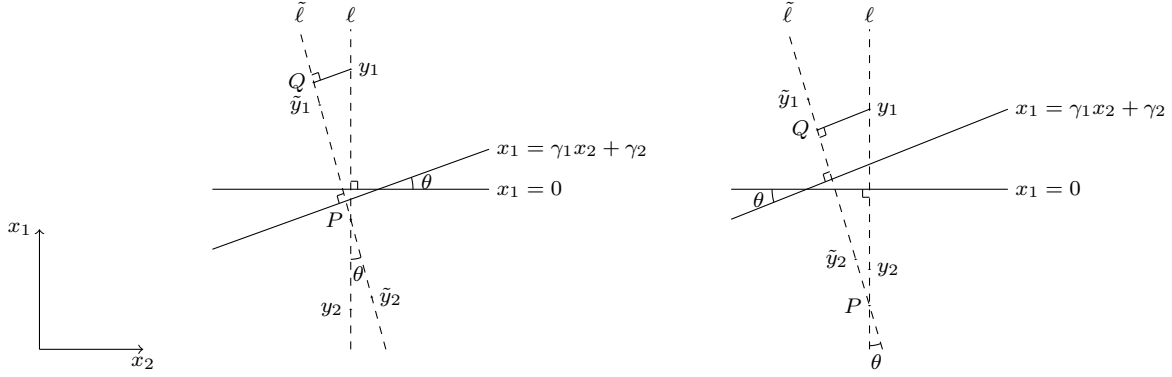
\begin{figure}
\centering
\setlength\figureheight{\textwidth/3}
\setlength\figurewidth{\textwidth}
\scriptsize
\begin{tikzpicture}
\begin{axis}[
width=\figurewidth,
height=\figureheight,
scale only axis,
ticks=none,
xmin=-5,
xmax=30,
ymin=-5,
ymax=5,
hide axis,
]
% Axis
\draw [->] (axis cs: -4,-4) -- (axis cs: -1,-4) node[below] {$x_2$};
\draw [->] (axis cs: -4,-4) -- (axis cs: -4,-1) node[left] {$x_1$};

% Case 1
\draw (axis cs: 1,0) -- (axis cs: 9,0) node[right] {$x_1 = 0$};
\draw[dashed] (axis cs: 5,-4) -- (axis cs: 5,4) node[above] {$\ell$};
\draw (axis cs: 5.2,0.2) -- (axis cs: 5,0.2);
\draw (axis cs: 5.2,0) -- (axis cs: 5.2,0.2);
\draw (axis cs: 1,-1.5) -- (axis cs: 9,1) node[right] {$x_1 = \gamma_1 x_2 + \gamma_2$};
\draw[dashed] (axis cs: 6.0156,-4) -- (axis cs: 3.5156,4) node[above] {$\tilde{\ell}$}; 
\draw (axis cs: 4.6668,-0.3542) -- (axis cs: 4.6071,-0.1633);
\draw (axis cs: 4.6071,-0.1633) -- (axis cs: 4.7980,-0.1036);
\filldraw (axis cs: 5,-3) circle [radius=0.1] node[left] {$y_2$};
\filldraw (axis cs: 5,3) circle [radius=0.1] node[right] {$y_1$};
\filldraw (axis cs: 5.6077,-2.6945) circle [radius=0.1] node[right] {$\tilde{y}_2$};
\filldraw (axis cs: 4.1077,2.1055) circle [radius=0.1] node[left] {$\tilde{y}_1$};
\filldraw (axis cs: 5,-0.75) circle [radius=0.1] node[left] {$P$};
%\draw[domain=270:287.35] plot (axis cs: {5+cos(\x)}, {-0.75+sin(\x)});
\addplot [
color=black,
solid,
forget plot
]
table[row sep=crcr]{
5.000 -1.750 \\
5.034 -1.749 \\
5.067 -1.748 \\
5.101 -1.745 \\
5.134 -1.741 \\
5.167 -1.736 \\
5.201 -1.730 \\
5.233 -1.722 \\
5.266 -1.714 \\
5.298 -1.705 \\
};
\draw (axis cs: 5.225,-2.1) node {$\theta$};
\filldraw (axis cs: 3.9324,2.6664) circle [radius=0.1] node[left] {$Q$};
\draw (axis cs: 3.9324,2.6664) -- (axis cs: 5,3);
\draw (axis cs: 4.1233,2.7261) -- (axis cs: 4.0636,2.9170);
\draw (axis cs: 4.0636,2.9170) -- (axis cs: 3.8727,2.8573);
%\draw[domain=0:17.35] plot (axis cs: {5.8+cos(\x)}, {sin(\x)});
\addplot [
color=black,
solid,
forget plot
]
table[row sep=crcr]{
6.800 0.000 \\
6.799 0.034 \\
6.798 0.067 \\
6.795 0.101 \\
6.791 0.134 \\
6.786 0.167 \\
6.780 0.201 \\
6.772 0.233 \\
6.764 0.266 \\
6.755 0.298 \\
};
\draw (axis cs: 7.2,0.2) node {$\theta$};

% Case 2
\draw (axis cs: 16,0) -- (axis cs: 24,0) node[right] {$x_1 = 0$};
\draw[dashed] (axis cs: 20,-4) -- (axis cs: 20,4) node[above] {$\ell$};
\draw (axis cs: 19.8,-0.2) -- (axis cs: 20,-0.2);
\draw (axis cs: 19.8,0) -- (axis cs: 19.8,-0.2);
\draw (axis cs: 16,-0.75) -- (axis cs: 24,2) node[right] {$x_1 = \gamma_1 x_2 + \gamma_2$};
\draw[dashed] (axis cs: 20.375,-4) -- (axis cs: 17.625,4) node[above] {$\tilde{\ell}$}; 
\draw (axis cs: 18.7244,0.1865) -- (axis cs: 18.6594,0.3756);
\draw (axis cs: 18.6594,0.3756) -- (axis cs: 18.8485,0.4406);
\filldraw (axis cs: 20,-2) circle [radius=0.1] node[right] {$y_2$};
\filldraw (axis cs: 20,2) circle [radius=0.1] node[right] {$y_1$};
\filldraw (axis cs: 19.6010,-1.7485) circle [radius=0.1] node[left] {$\tilde{y}_2$};
\filldraw (axis cs: 18.2260,2.2515) circle [radius=0.1] node[left] {$\tilde{y}_1$};
\filldraw (axis cs: 20,-2.9093) circle [radius=0.1] node[left] {$P$};
%\draw[domain=270:288.97] plot (axis cs: {20+cos(\x)}, {-2.9093+sin(\x)});
\addplot [
color=black,
solid,
forget plot
]
table[row sep=crcr]{
20.000 -3.909 \\
20.037 -3.909 \\
20.074 -3.907 \\
20.110 -3.903 \\
20.147 -3.898 \\
20.183 -3.892 \\
20.219 -3.885 \\
20.255 -3.876 \\
20.290 -3.866 \\
20.325 -3.855 \\
};
\draw (axis cs: 20.2,-4.3) node {$\theta$};
\filldraw (axis cs: 18.4908,1.4812) circle [radius=0.1] node[left] {$Q$};
\draw (axis cs: 18.4908,1.4812) -- (axis cs: 20,2);
\draw (axis cs: 18.6799,1.5462) -- (axis cs: 18.7449,1.3571);
\draw (axis cs:  18.7449,1.3571) -- (axis cs: 18.5558,1.2921);
%\draw[domain=180:198.97] plot (axis cs: {18.1818+cos(\x)}, {sin(\x)});
\addplot [
color=black,
solid,
forget plot
]
table[row sep=crcr]{
17.182 0.000 \\
17.182 -0.037 \\
17.185 -0.074 \\
17.188 -0.110 \\
17.193 -0.147 \\
17.199 -0.183 \\
17.206 -0.219 \\
17.215 -0.255 \\
17.225 -0.290 \\
17.236 -0.325 \\
};
\draw (axis cs: 16.8,-0.25) node {$\theta$};
\end{axis}
\end{tikzpicture}
\caption{
%The 2D geometry in the proof of Lemma~\ref{lem:Rate:GatDerFinfty} has two possibilities.
%The first case is when the intersection of $\ell$ and $\tilde{\ell}$ is between $y_1$ and $y_2$ shown on the left.
%The other case is shown on the right.
The geometry considered in the proof of Lemma~\ref{lem:Rate:GatDerFinfty} admits two cases: in the first (left) the intersection of $l$ and $\tilde{l}$ lies between $y_1$ and $y_2$; in the second (right) it does not.
}
\label{fig:Rate:2DGeom}
\end{figure}

The first case is similar.
Assume that $|y_1-P| \geq |y_2-P|$ then $|y_1-P|\geq \frac{\delta}{2}$.
Let $Q$ be the closest point on $\tilde{\ell}$ to $y_1$ then $|Q-y_1|\leq |y_1 -\tilde{y}_1|\leq Cr$ and $\widehat{QPy_1} = \theta$, $\widehat{y_1 QP}=\frac{\pi}{2}$.
In particular $\sin\theta = \frac{|Q-y_1|}{|y_1-P|} \leq \frac{2Cr}{\delta}$.

%Finding the maximum angle becomes a 2D problem and it is not difficult to show that $\sin\theta \leq \frac{2Cr}{\delta}$.
In both cases $\sin\theta \leq \frac{2Cr}{\delta}$ which implies
\[ d_H(B_{y_1,y_2},B_{\tilde{y}_1,\tilde{y}_2}) \leq d_H(\partial B_{y_1,y_2},\partial B_{\tilde{y}_1,\tilde{y}_2}) \leq rC + \frac{4MCr}{\delta}. \]
%The case $d=1$ is trivial as one does not need to consider rotations, so $d_H(B_{y_1,y_2},B_{\tilde{y}_1,\tilde{y}_2}) \leq rC$.

%More precisely at each $t$ we have $j_r(t,y)=j(t,y)$ for all $y$ with $\text{dist}(y,B(t)) > r\max_i \|\nu_i\|_{L^\infty}$ where $B(t)$ is the set of boundary points:
Let
\[ B(t) = \left\{ y \in\mathbb{R}^d \, : \, j(t,y) \text{ is not uniquely defined} \right\} \]
and $X(r,t) = \left\{y \in B(0,r^{-\beta})\, :\, \text{dist}(y,B(t))\leq \|\nu\|_{(L^\infty)^k} \left(r + \frac{4r^{1-\beta}}{\delta}\right)\right\}$.
By the previous calculation with $C = \|\nu\|_{(L^\infty)^k}$ and $M = r^{-\beta}$, if $j_r(t,y)\neq j(t,y)$ then $\mathrm{dist}(y,B(t)) \leq rC+\frac{4Mr}{\delta} = \|\nu\|_{(L^\infty)^k} \left(r + \frac{4r^{1-\beta}}{\delta}\right)$.
And therefore if $y\not\in X(r,t)$ then $j_r(t,y) = j(t,y)$.

%The set $X(r,t)$ is contained in the union of $k(k-1)$ $d$-dimensional hyper-cuboids of size~$\left(2r^{-\beta}\right)^{d-1}\times \left(\|\nu\|_{(L^\infty)^k} \left(r + \frac{4r^{1-\beta}}{\delta}\right)\right)$.
We now partition $X(r,t)$ into $\lceil 2r^{-\beta-1}\rceil$ subsets (where $\lceil t\rceil$ is the smallest integer greater than or equal to $t$) by defining
\[ B^m_{y_1,y_2} = \left\{ y \in B_{y_1,y_2} \, : \, \left|y-\frac{y_1+y_2}{2}\right| \in [(m-1)r, mr] \right\} \]
and
\begin{align*}
X_m(r,t) & = \Bigg\{ y \in X(r,t) \, : \, \exists i,j \text{ with } \mathrm{dist}(y,B^m_{\mu_i(t),\mu_j(t)}) \leq \left(2 \|\nu\|_{(L^\infty)^k} \left(r + \frac{2r^{1-\beta}}{\delta}\right)\right) \\
 & \quad \quad \quad \quad \text{and } \mathrm{dist}(y,B^m_{\mu_i(t),\mu_j(t)}) \leq \mathrm{dist}(y,B^{m'}_{\mu_i(t),\mu_j(t)}) \text{ for all } m'\neq m \Bigg\}.
\end{align*}
So $X(r,t) \subset \cup_{m=1}^{\lceil 2r^{-\beta-1}\rceil} X_m(r,t)$ (assuming $r$ is sufficiently small so that $\|\nu\|_{(L^\infty)^k}\leq r^{-\beta}$).
This implies
\begin{align*}
& \left| \int_{|y|\leq r^{-\beta}} \left( 2y - \mu_{j(t,y)}(t) - \mu_{j_r(t,y)}(t) \right) \cdot \left( \mu_{j(t,y)}(t) - \mu_{j_r(t,y)}(t) \right) \; \phi_Y(y|t) \; \mathrm{d} y \right| \\
& \quad \quad = \left| \int_{X(r,t)} \left( 2y - \mu_{j(t,y)}(t) - \mu_{j_r(t,y)}(t) \right) \cdot \left( \mu_{j(t,y)}(t) - \mu_{j_r(t,y)}(t) \right) \; \phi_Y(y|t) \; \mathrm{d} y \right| \\
& \quad \quad \leq 2 \sum_{m=1}^{\lceil 2r^{-\beta-1} \rceil} \left(mr + \|\nu\|_{(L^\infty)^k} \left(r + \frac{4r^{1-\beta}}{\delta} \right) \right) \int_{X_m(r,t)} |\mu_{j(t,y)}(t)-\mu_{j_r(t,y)}(t)| \phi_Y(y|t) \, \mathrm{d}y.
\end{align*}

Now if $y\in X_m(r,t)$ then $\left|y-\frac{\mu_j(t)+\mu_i(t)}{2}\right| \geq (m-1)r$ for some $i,j$ and therefore $|y| \geq (m-1) r - A$ where $\|\mu\|_{(L^\infty)^k}\leq A$.
In particular
\[ \phi_Y(y|t) \leq \left\{ \begin{array}{ll} c_1 (m-1-A)^\alpha & \text{if } m > A+1 \\ \|\phi_Y\|_{L^\infty} & \text{else.} \end{array} \right. \]
Note that
\begin{align*}
\mathrm{Vol}(X_m(r,t)) & \leq k(k-1) \left[ \mathrm{Vol}_{d-1}(B(0,mr))-\mathrm{Vol}_{d-1}(B(0,(m-1)r) \right] \left[ \|\nu\|_{(L^\infty)^k} \left( r + \frac{4r^{1-\beta}}{\delta} \right) \right] \\
 & \lesssim m^{d-1} r^{d-\beta}.
\end{align*}
Therefore
\begin{align*}
& \frac{1}{r} \left| \int_{|y|\leq r^{-\beta}} \left( 2y - \mu_{j(t,y)}(t) - \mu_{j_r(t,y)}(t) \right) \cdot \left( \mu_{j(t,y)}(t) - \mu_{j_r(t,y)}(t) \right) \; \phi_Y(y|t) \; \mathrm{d} y \right| \\
& \quad \quad \leq \frac{2 \|\mu\|_{(L^\infty)^k}}{r} \sum_{m=1}^{\lceil 2r^{-\beta-1} \rceil} \left( mr + \|\nu\|_{(L^\infty)^k} \left( r + \frac{4r^{1-\beta}}{\delta} \right) \right) \int_{X_m(r,t)} \phi_Y(y|t) \, \mathrm{d} y \\
& \quad \quad \leq \frac{2 \|\mu\|_{(L^\infty)^k} \|\phi_Y\|_{L^\infty}}{r} \sum_{m=1}^{A+1} \left( mr + \|\nu\|_{(L^\infty)^k} \left( r + \frac{4r^{1-\beta}}{\delta} \right) \right) \mathrm{Vol}(X_m(r,t)) \\
& \quad \quad \quad \quad + \frac{2 c_1 \|\mu\|_{(L^\infty)^k}}{r} \sum_{m=A+2}^{\lceil 2r^{-\beta-1} \rceil} \left( mr + \|\nu\|_{(L^\infty)^k} \left( r + \frac{4r^{1-\beta}}{\delta} \right) \right) (m-1-A)^\alpha \mathrm{Vol}(X_m(r,t)) \\
& \quad \quad \lesssim \frac{1}{r} \sum_{m=1}^{A+1} (rm + r^{1-\beta}) m^{d-1} r^{d-\beta} + \frac{1}{r} \sum_{m=A+2}^{\lceil 2r^{-\beta-1} \rceil} (rm + r^{1-\beta}) (m-1-A)^\alpha m^{d-1} r^{d-\beta} \\
& \quad \quad \lesssim r^{d-2\beta} + r^{d-2\beta} \sum_{m=1}^\infty m^{d+\alpha} \\
& \quad \quad = O(r^{d-2\beta})
\end{align*}
with the above following as $r^{d-\beta}$ is dominated by $r^{d-2\beta}$ as $r\to 0$.
Since $d-2\beta\geq 2(1-\beta)>0$ then the above is o(1).

Hence
\begin{align*}
& \frac{1}{r} \left| \int_{\mathbb{R}^d} \left| y- \mu_{j_r(t,y)}(t) \right|^2 - \left| y- \mu_{j(t,y)}(t) \right|^2 \; \phi_Y(y|t) \; \mathrm{d} y \right| \\
& \quad \quad \quad \quad \leq \frac{1}{r} \left| \int_{|y|\leq r^{-\beta}} \left| y- \mu_{j_r(t,y)}(t) \right|^2 - \left| y- \mu_{j(t,y)}(t) \right|^2 \; \phi_Y(y|t) \; \mathrm{d} y \right| + o(1) \quad \text{by~\eqref{eq:Rate:ResToBall}} \\
& \quad \quad \quad \quad = \frac{1}{r} \left| \int_{|y|\leq r^{-\beta}} \left( 2y - \mu_{j(t,y)}(t) - \mu_{j_r(t,y)}(t) \right) \cdot \left( \mu_{j(t,y)}(t) - \mu_{j_r(t,y)}(t) \right) \; \phi_Y(y|t) \; \mathrm{d} y \right| + o(1)
\end{align*}
which converges (uniformly in $t$) to zero.

%where the last line follows from: if $j(t,y)\neq j_r(t,y)$ then
%\[ \mathrm{dist}\left(\frac{\mu_{j(t,y)}(t)+\mu_{j_r(t,y)}(t)}{2},B(t)\right)\leq r \max_i \|\nu_i\|_{L^\infty}. \]
%Since $\phi_Y$ is integrable there exists $\gamma\in \left(\frac{d-3}{d-1},1\right)$ and a sequence $L_r\to 0$ such that
%\[ \int_{\mathbb{R}^d\setminus B(0,r^{\gamma-1})} \phi_Y(y|t) \; \mathrm{d} y \leq L_r \quad \quad \text{for all } t \in [0,1]. \]
%Since $B(t)$ consists of subsets of at most $k-1$ $(d-1)$-dimensional
%hyperplanes then for any $\alpha$, the $(d-1)$-dimensional Lebesgue measure of its
%intersection with a ball may be straightforwardly bounded: $\mathcal{L}^{d-1}(B(t) \cap B(0,\alpha)) \leq 2^{d-1}(k-1) \alpha^{d-1}$.
%A direct consequence of which is that:
%\[ \mathcal{L}^d(X(r,t)\cap B(0,r^{\gamma-1})) \leq \max_i \|\nu_i\|_{L^\infty} 2^d(k-1) r^{(\gamma-1)(d-1)+1}. \]
%Hence
%\begin{align}
%& \frac{1}{r} \int_{\mathbb{R}^d} \left| y- \mu_{j_r(t,y)}(t) \right|^2 - \left| y- \mu_{j(t,y)}(t) \right|^2 \; \phi_Y(y|t) \; \mathrm{d} y \notag \\
%& \leq 4 \left( \max_i \| \nu_i\|_{L^\infty} \right) \left(\max_j \|\mu_j\|_{L^\infty} \right) \left( \int_{X(r,t) \cap B(0,r^{\gamma-1})} \phi_Y(y|t) \; \mathrm{d} y + \int_{\mathbb{R}^d\setminus B(0,r^{\gamma-1})} \phi_Y(y|t) \; \mathrm{d} y \right) \notag \\
%& \leq 4 \left( \max_i \| \nu_i\|_{L^\infty} \right) \left(\max_j \|\mu_j\|_{L^\infty} \right) \left( 2^d(k-1) r^{(\gamma-1)(d-1)+2} \max_i \|\nu_i\|_{L^\infty} \|\phi_Y\|_{L^\infty} + L_r \right) \notag \\
%& \to 0 \label{eq:Rate:UnifConvt}
%\end{align}
%since $(\gamma-1)(d-1)+3 >0$.

Therefore
\begin{align*}
\partial f_\infty(\mu;\nu) & = \lim_{r\to 0} \frac{f_\infty(\mu+r\nu) - f_\infty(\mu)}{r} \\
 & = \lim_{r\to 0} \frac{1}{r} \Bigg\{ \int_0^1 \int_{\mathbb{R}^d} \Bigg( |y-\mu_{j_r(t,y)}(t)|^2 - |y-\mu_{j(t,y)}(t)|^2 + r^2 |\nu_{j_r(t,y)}(t)|^2 \\
 & \quad \quad \quad \quad - 2r\left(y-\mu_{j_r(t,y)}(t)\right) \cdot \nu_{j_r(t,y)}(t) \Bigg) \phi_Y(y|t) \phi_T(t) \; \text{d} y \text{d} t \\
 & \quad \quad \quad \quad +  \lambda \sum_{j=1}^k \left(2r (\nabla^s\nu_j,\nabla^s \mu_j) + r^2 \|\nabla^s\nu_j\|_{L^2}^2 \right) \Bigg\} \\
 & = -2\int_0^1 \int_{\mathbb{R}^d} \left(y-\mu_{j(t,y)}(t)\right) \cdot \nu_{j(t,y)}(t) \phi_Y(y|t) \phi_T(t) \; \text{d} y \text{d} t \\
 & \quad \quad \quad \quad + 2 \lambda \sum_{j=1}^k (\nabla^s \nu_j,\nabla^s \mu_j)
\end{align*}
by the dominated convergence theorem. % and the uniform in $t$ convergence of $\int_{\mathbb{R}^d} |y-\mu_{j_r(t,y)}(t)|^2 - |y-\mu_{j(t,y)}(t)|^2 \phi_Y(y|t) \; \mathrm{d} y$ implied by~\eqref{eq:Rate:UnifConvt}.
\end{proof}

\begin{lemma} \label{lem:Rate:Gat2ndDerFinfty}
%If, in addition to the assumptions of Lemma~\ref{lem:Rate:GatDerFinfty}, we further adopt Assumption~4 then we have 
Under the same conditions as Lemma~\ref{lem:Rate:GatDerFinfty} we have
\[ \partial^2_-f_\infty(\mu;\nu,\nu) \geq 2\lambda \|\nabla^s\nu\|_{(L^2)^k}^2 - 2 \|\nu\|_{(L^\infty)^k}^2. \]
%\textbf{Did not need assumption 4, only true if $\alpha < -4 - \frac{2}{d-2}$, not true for $d=1,2$???}
%for constants $\kappa_1$ and $\kappa_2$.
\end{lemma}

\begin{proof}
The proof is similar to that of Lemma~\ref{lem:Rate:GatDerFinfty} so we only sketch the details.
The key step is in showing the following limit converges to zero
\begin{align*}
& \limsup_{r\to 0} \frac{1}{r} \int_0^1 \int_{\mathbb{R}^d} \Big\{ \left(\mu_{j_r(t,y)}(t) - y \right) \cdot \nu_{j_r(t,y)}(t) - \left(\mu_{j(t,y)}(t) - y \right) \cdot \nu_{j(t,y)}(t) \Big\} \phi_Y(y|t) \phi_T(t) \; \mathrm{d} y \; \mathrm{d} t \\
& \quad \quad \leq 2 \|\mu\|_{(L^\infty)^k} \|\nu\|_{(L^\infty)^k} \limsup_{r\to 0} \frac{1}{r} \int_0^1 \int_{j_r\neq j} \phi_Y(y|t) \phi_T(t) \; \mathrm{d} y \; \mathrm{d} t \\
& \quad \quad  \quad \quad + 2 \|\nu\|_{(L^\infty)^k} \limsup_{r\to 0} \frac{1}{r} \int_0^1 \int_{j_r\neq j} |y| \phi_Y(y|t) \phi_T(t) \; \mathrm{d} y \; \mathrm{d} t.
\end{align*}
%where $X(r,t)$ contains the set where
%We now show that $\limsup_{r\to 0} \frac{1}{r} \int_0^1 \int_{j_r\neq j} |y| \phi_Y(y|t) \phi_T(t) \; \mathrm{d} y \; \mathrm{d} t = 0$.
As in the proof of Lemma~\ref{lem:Rate:GatDerFinfty} we divide $\mathbb{R}^d = B(0,r^{-\beta})\cup (\mathbb{R}^d\setminus B(0,r^{-\beta}))$ and recall that $X(r,t)$ contains the set where $j_r(t,y)\neq j(t,y)$ in the ball $B(0,r^{-\beta})$ and $X(r,t) \subset \cup_{m=1}^{\lceil 2r^{-\beta-1} \rceil} X_m(r,t)$ with $\mathrm{Vol}(X_m(r,t)) = O(m^{d-1}r^{d-\beta})$.
The limit
\[ \frac{1}{r} \int_{|y|\geq r^{-\beta}} \left( |y| + 1 \right) \phi_Y(y|t) \phi_T(t) \, \mathrm{d} y \, \mathrm{d} t \to 0 \]
as in the proof of Lemma~\ref{lem:Rate:GatDerFinfty}.
Now,
\begin{align*}
& \int_{j_r\neq j, |y|\leq r^{-\beta}} (|y|+1) \phi_Y(y|t) \; \mathrm{d} y \\
 & \quad \leq \int_{X(r,t)} (|y|+1) \phi_Y(y|t) \; \mathrm{d} y \\
 & \quad \leq \sum_{m=1}^{A+1} \int_{X_m(r,t)} (|y|+1) \phi_Y(y|t) \; \mathrm{d} y + c_1 \sum_{m=A+2}^{\lceil 2r^{-\beta-1} \rceil} \int_{X_m(r,t)} (|y|+1)|y|^\alpha \; \mathrm{d} y \\
 & \quad \leq \sum_{m=1}^{A+1} \|\phi_Y\|_{L^\infty} (A+mr+1) \mathrm{Vol}(X_m(r,t)) + c_1 \sum_{m=A+2}^{\lceil 2r^{-\beta-1} \rceil} (m-A)\left(m-1-A\right)^\alpha \mathrm{Vol}(X_m(r,t)) \\
 & \quad = O(r^{d-\beta}).
\end{align*}
%  & \leq \|\phi_Y\|_{L^\infty} \mathrm{Vol}(X_1(r,t)) + c_1 \|\phi_Y\|_{L^\infty} \sum_{m=2}^A \mathrm{Vol}(X_m(r,t)) \\
Since $d-\beta \geq 2-\beta > 1$ then the above limit is $o(r)$.
\end{proof}

We now consider $Y_n$.
In particular we want to bound $\partial Y_n(\mu^\infty;\mu^n-\mu^\infty)$.

\begin{lemma} \label{lem:Rate:d2Yn}
Define $f_n,f_\infty:\Theta\to\mathbb{R}$ by~\eqref{eq:Intro:fn} and~\eqref{eq:Intro:finfty} respectively where $\Theta$ is given by~\eqref{eq:Conv:Theta}.
Take Assumptions~1, 2 and~4 and define
\[ Y_n: \Theta \to \mathbb{R}, \quad \quad \quad \quad Y_n(\mu) = \sqrt{n} \left(f_n(\mu) - f_\infty(\mu) \right). \]
Then for $\mu\in\Theta$, $\nu\in (H^s)^k$ we have that $Y_n$ is G\^ateaux differentiable at $\mu$ in the direction $\nu$ with
\begin{align*}
\partial Y_n (\mu;\nu) & = 2\sqrt{n} \left( \int_0^1 \int_\mathbb{R} \left( y-\mu_{j(t,y)}(t) \right) \cdot \nu_{j(t,y)}(t) \phi_Y(y|t) \phi_T(t) \; \text{d}y \text{d} t \right. \\
 & \left. - \frac{1}{n} \sum_{i=1}^n \left(y_i-\mu_{j(t_i,y_i)}(t_i)\right) \cdot \nu_{j(t_i,y_i)}(t_i)  \right)
\end{align*}
where $j(t,y)$ is defined by~\eqref{eq:Rate:jty}.
Furthermore, for a sequence $\nu^n$ with
\[ \|\nu^n\|_{(L^2)^k} = o_p(1) \quad \quad \text{and} \quad \quad \|\nu^n\|_{(H^s)^k}=O_p(1) \]
we have $\partial Y_n(\mu;\nu^n) = O_p(\|\nu^n\|_{(L^2)^k})$.
\end{lemma}

\begin{proof}
Calculating the G\^ateaux derivative is similar to Lemma~\ref{lem:Rate:GatDerFinfty} and is omitted.
By linearity and continuity of $\partial Y_n$ we can write
\[ \partial Y_n\left(\mu;\frac{\nu^n}{\|\nu^n\|_{(L^2)^k}}\right) = \sum_m \frac{(\nu^n,e_m)}{\|\nu^n\|_{(L^2)^k}} \partial Y_n(\mu;e_m) \]
where $e_m$ is the Fourier basis for $(L^2)^k$ (we assume $e_m=(\hat{e}_{m_1},\dots,\hat{e}_{m_k})$ where $\hat{e}_m$ is the Fourier basis for $L^2$).
Let $V_m=\mathbb{E}\left(\partial Y_n(\mu;e_m) \right)^2$ and $Z_i=(y_i-\mu_{j(t_i,y_i)}(t_i))\cdot \hat{e}_{m_j(t_i,y_i)}$, then
\begin{align*}
V_m & = \frac{4}{n} \mathbb{E} \left(\sum_{i=1}^n(Z_i - \mathbb{E}Z_i) \right)^2 \\
 & = 4\mathbb{E}\left(Z_1-\mathbb{E} Z_1\right)^2 \\
 & \leq 4 \left( 4\sum_{j=1}^k \sum_{l=1}^k \|\mu^\dagger_l - \mu_j\|_{L^\infty} + \mathbb{E} \left|\epsilon_1\right|^2 \right) =:C.
\end{align*}
By Assumptions 1 and 2 and since $\mu\in (L^\infty)^k$ (by the embedding of $(H^s)^k$ into $(L^\infty)^k)$) $C$ is finite.
Therefore, %by Chebyshev's inequality for any $M>0$
%\[ \mathbb{P}\left( |\partial Y_n(\mu;e_m)| \geq M \right) \leq \frac{C}{M^2}. \]
%Hence
\begin{align*}
\mathbb{P} \left(\left|\partial Y_n\left(\mu;\frac{\nu^n}{\|\nu^n\|_{(L^2)^k}}\right)\right| \geq M \right) & \leq \frac{1}{M} \mathbb{E}\left(\left|\partial Y_n\left(\mu;\frac{\nu^n}{\|\nu^n\|_{(L^2)^k}}\right)\right| \right) \quad \text{by Markov's inequality} \\
 & \leq \sum_m \frac{|(\nu^n,e_m)|}{\|\nu^n\|_{(L^2)^k}} \frac{1}{M} \mathbb{E}\left( |\partial Y_n(\mu;e_m)| \right) \\
 & \leq \sum_m \frac{|(\nu^n,e_m)|}{\|\nu^n\|_{(L^2)^k}} \frac{\sqrt{V_m}}{M} \quad \text{by H\"older's inequality} \\
 & \leq \frac{\sqrt{C}}{M}.
\end{align*}
Which implies $\partial Y_n\left(\mu;\frac{\nu^n}{\|\nu^n\|_{(L^2)^k}}\right)=O_p(1)$.
\end{proof}

We now have the necessary pieces in place to prove Theorem~\ref{thm:Rate:StrongConv} and Corollary~\ref{cor:Rate:Rate}.

\begin{proof}[Proof of Theorem~\ref{thm:Rate:StrongConv}]
By Theorem~\ref{thm:Conv:Conv} we have that (up to subsequences) $\|\mu^n- \mu^\infty\|_{(L^2)^k}=o_p(1)$, $\|\mu^n- \mu^\infty\|_{(L^\infty)^k}=o_p(1)$ and $\|\mu^n\|_{(H^s)^k}=O_p(1)$.

By Theorem~\ref{thm:Prelim:GatDer:Taylor}, for some $t\in [0,1]$, we have
\begin{align*}
f_\infty(\mu^n) & \geq f_\infty(\mu^\infty) + \partial f_\infty\left(\mu^\infty;\mu^n-\mu^\infty\right) + \frac{1}{2} \partial^2_- f_\infty\left((1-t)\mu^\infty+t\mu^n;\mu^n-\mu^\infty\right) \\
 & \geq f_\infty(\mu^\infty) + 2\lambda \|\nabla^s (\mu^n-\mu^\infty)\|_{(L^2)^k}^2 - 2 \|\mu^n-\mu^\infty\|_{(L^\infty)^k}^2
\end{align*}
after applying Lemma~\ref{lem:Rate:Gat2ndDerFinfty} and since $\mu^\infty$ minimizes $f_\infty$ the first derivative must be zero.

Similarly, and using Lemma~\ref{lem:Rate:d2Yn},
\[ Y_n(\mu^n) = Y_n(\mu^\infty) + O_p\left(\partial Y_n\left(\mu^\infty;\mu^n-\mu^\infty\right)\right) = Y_n(\mu^\infty) + O_p\left(\|\mu^n-\mu^\infty\|_{(L^2)^k}\right). \]
From the definition of $Y_n$ we also have
\[ f_n(\mu^n) = f_\infty(\mu^n) + \frac{1}{\sqrt{n}} Y_n(\mu^n). \]
Substituting into the above we obtain
\begin{align*}
f_n(\mu^n) & \geq f_\infty(\mu^\infty) + \frac{1}{\sqrt{n}} Y_n(\mu^n) + 2\lambda \|\nabla^s \left(\mu^n-\mu^\infty\right)\|_{(L^2)^k}^2 - 2 \|\mu^n-\mu^\infty\|_{(L^\infty)^k}^2 \\
 & = f_\infty(\mu^\infty) + \frac{1}{\sqrt{n}} Y_n(\mu^\infty) + O_p\left(\frac{\|\mu^n-\mu^\infty\|_{(L^2)^k}}{\sqrt{n}}\right) + 2\lambda \|\nabla^s \left(\mu^n-\mu^\infty\right)\|_{(L^2)^k}^2 \\
 & \quad \quad \quad \quad \quad \quad - 2 \|\mu^n-\mu^\infty\|_{(L^\infty)^k}^2 \\
 & = f_n(\mu^\infty) + O_p\left(\frac{\|\mu^n - \mu^\infty\|_{(L^2)^k}}{\sqrt{n}} + \|\mu^n-\mu^\infty\|_{(L^\infty)^k}^2\right) + 2\lambda \|\nabla^s\left(\mu^n-\mu^\infty\right)\|_{(L^2)^k}^2.
\end{align*}
Rearranging and using $f_n(\mu^n) \leq f_n(\mu^\infty)$ we have
\begin{align*}
2\lambda \|\nabla^s \left(\mu^n-\mu^\infty\right)\|_{(L^2)^k}^2 & \leq \left(f_n(\mu^n) - f_n(\mu^\infty)\right) + O_p\left(\frac{\|\mu^n - \mu^\infty\|_{(L^2)^k}}{\sqrt{n}}+\|\mu^n-\mu^\infty\|_{(L^\infty)^k}^2\right) \\
 & \leq O_p\left(\frac{\|\mu^n - \mu^\infty\|_{(L^2)^k}}{\sqrt{n}}+\|\mu^n-\mu^\infty\|_{(L^\infty)^k}^2\right).
\end{align*}
We have shown, via Theorem \ref{thm:Conv:Conv}, that $\|\nabla^s\left(\mu^n-\mu^\infty\right)\|_{(L^2)^k}\to 0$ and therefore $\mu^n\to \mu$ strongly in $H^s$ and in probability.
\end{proof}

\begin{proof}[Proof of Corollary~\ref{cor:Rate:Rate}]
The proof is similar to the proof of Theorem~\ref{thm:Rate:StrongConv} since
\begin{align*}
f_\infty(\mu^n) & \geq f_\infty(\mu^\infty) + \partial f_\infty\left(\mu^\infty;\mu^n-\mu^\infty\right) + \int_0^1 (1-t) \partial^2_- f_\infty\left((1-t)\mu^\infty+t\mu^n;\mu^n-\mu^\infty\right) \\
 & \geq f_\infty(\mu^\infty) + \kappa \|\mu^n-\mu^\infty\|_{(H^s)^k}^2.
\end{align*}
One can then show
\begin{align*}
f_\infty(\mu^n) - f_\infty(\mu^\infty) & = f_n(\mu^n) - \frac{1}{\sqrt{n}} Y_n(\mu^n) - f_\infty(\mu^\infty) \\
 & = f_n(\mu^n) - f_\infty(\mu^\infty) - \frac{1}{\sqrt{n}} Y_n(\mu^\infty) + O_p\left(\frac{\|\mu^n-\mu^\infty\|_{(L^2)^k}}{\sqrt{n}}\right) \\
 & = f_n(\mu^n) - f_n(\mu^\infty) + O_p\left(\frac{\|\mu^n-\mu^\infty\|_{(L^2)^k}}{\sqrt{n}}\right) \\
 & \leq O_p\left(\frac{\|\mu^n-\mu^\infty\|_{(L^2)^k}}{\sqrt{n}}\right)
\end{align*}
Hence,
\[ \kappa \|\mu^n-\mu^\infty\|_{(H^s)^k} \|\mu^n-\mu^\infty\|_{(L^2)^k} \leq \kappa \|\mu^n - \mu^\infty\|_{(H^s)^k}^2 \leq O_p\left(\frac{\|\mu^n-\mu^\infty\|_{(L^2)^k}}{\sqrt{n}}\right). \]
Dividing by $\|\mu^n-\mu^\infty\|_{(L^2)^k}$ completes the proof.
\end{proof}

\section*{Acknowledgements}

Part of this work was completed whilst MT was part of MASDOC at the University of Warwick and was supported by an EPSRC Industrial CASE Award PhD Studentship with Selex ES Ltd.
The authors would also like to thank Neil Cade (Selex ES Ltd.) and Florian Theil (Warwick University) whose discussions enhanced this paper as well as Riccardo Cristoferi for feedback on an earlier version of the manuscript.
The authors also gratefully acknowledge the feedback of the referee whose comments significantly improved this manuscript.

%\appendix
%\appendixpage

%\begin{lemma}
%\label{lem:app:2ndGatDerCarefull}
%Under the same conditions as Lemma~??? we have that
%\[ d_-^2 f_\infty(\mu;\nu) \geq \kappa \|\nu\|_{(L^2)^k}^2 + 2\lambda \|\nabla \nu\|_{(L^2)^k}^2 - B(\mu,\nu) \]
%where $B(\mu,\nu)$ is defined by~???.
%\end{lemma}

%\nocite{*}
\bibliographystyle{plain}
\bibliography{references}

\end{document}